\documentclass{article}
\usepackage{amsmath,amssymb,amsthm}
\usepackage{graphicx,subfigure,float,url,color}
\usepackage{mathrsfs}
\usepackage[colorlinks=true]{hyperref}
\usepackage{pdfsync}
\usepackage{enumerate,enumitem}

\def\R{\textrm{I\kern-0.21emR}}
\def\N{\textrm{I\kern-0.21emN}}

\def\Tv{\mathbb{|\kern-0.1em|\kern-0.1em|}}

\renewcommand{\geq}{\geqslant}
\renewcommand{\leq}{\leqslant}

\newtheorem{theorem}{Theorem}
\newtheorem{proposition}{Proposition}

\newtheorem{lemma}{Lemma}
\theoremstyle{definition}
\theoremstyle{definition}\newtheorem{remark}{Remark}

\newcommand{\Real}{\mathrm{Re}}

\begin{document}
\title{Stabilization of infinite-dimensional linear control systems by POD reduced-order Riccati feedback}

\author{
Emmanuel Tr\'{e}lat\thanks{Sorbonne Universit\'e, Universit\'e Paris-Diderot SPC, CNRS, Inria, Laboratoire Jacques-Louis Lions, \'equipe CAGE, F-75005 Paris (\texttt{emmanuel.trelat@sorbonne-universite.fr})}
\quad
Gengsheng Wang\thanks{Center for Applied Mathematics, Tianjin University, Tianjin, 300072, China (\texttt{wanggs62@yeah.net}). The author was partially supported by  11571264. }
\quad
Yashan Xu\thanks{School of Mathematical Sciences, Fudan University, KLMNS, Shanghai 200433, China (\texttt{yashanxu@fudan.edu.cn}).}}

\date{}

\maketitle

\begin{abstract}
There exist many ways to stabilize an infinite-dimensional linear autonomous control systems when it is possible. Anyway, finding an exponentially stabilizing feedback control that is as simple as possible may be a challenge. The Riccati theory provides a nice feedback control but may be computationally demanding when considering a discretization scheme. Proper Orthogonal Decomposition (POD) offers a popular way to reduce large-dimensional systems. In the present paper, we establish that, under appropriate spectral assumptions, an exponentially stabilizing feedback Riccati control designed from a POD finite-dimensional approximation of the system stabilizes as well the infinite-dimensional control system.
 \end{abstract}

\textbf{Keywords:} Feedback stabilization, Proper Orthogonal Decomposition (POD), Riccati theory, linear quadratic optimal control.

\medskip

\textbf{AMS subject classifications:} 34H15, 35K05,  49M27, 49N10.

\section{Introduction and main result}

Stabilization of linear autonomous control systems is classically done in finite dimension by pole-shifting or by Riccati theory (see, e.g., \cite{kwak, LeeMarkus, Sontag, Trelat}). In infinite dimension, pole-shifting may be used for some appropriate classes of systems (see \cite{Barbu, CoronTrelat1, CoronTrelat2}, see also \cite[page 711]{Russell} and \cite[Chapter 3]{WangXu}), but such approaches rely on spectral considerations and in practice require the numerical computation of eigenelements, which may be hard in general. Riccati theory has also been much explored in infinite dimension (see, e.g., \cite{CurtainZwart, LT1, LT2, Z} and provides a powerful way for stabilizing a linear control system. Anyway, in practice, computing an approximation of the Riccati operator requires to consider a numerical approximation scheme and to compute the solution of a high-dimensional algebraic Riccati equation (see, e.g., \cite{BanksKunisch, KappelSalamon, LiuZheng, LT1, LT2} for convergence results for space semi-discretizations of the Riccati procedure, see also the survey \cite{trelat_stab}), which raises also a number of numerical difficulties.

Given these facts, it appears interesting to use dimension reduction procedures.
Indeed, model reduction can generate low-dimensional models for which one may expect reasonable performances for stabilization issues while keeping a computationally tractable numerical problem.
Proper Orthogonal Decomposition (POD) is a popular reduction model approach and can be used to generate, from a finite number $n$ of snapshots, a reduced-order control system in dimension $n$, approximating in the least square sense the initial infinite-dimensional system. Such an approach is completely general and does not consist of computing eigenelements (POD does not see eigenvectors). It is then natural to expect that, if $n$ is large enough, then a linear stabilizing feedback computed from the $n$-dimensional reduced-order control system, stabilizes as well the whole infinite-dimensional control system. Proving that this assertion holds true under appropriate assumptions is the objective of this paper: we prove that a low-order feedback control obtained by the Riccati procedure applied to a POD reduced-order model suffices to stabilize the complete infinite-dimensional control system.

The idea of using POD as a way to efficiently stabilize infinite-dimensional control systems, such as controlled PDEs, by means of a low-order feedback control, has been implemented in \cite{AtwellKing, Kunisch1, Kunisch2}, where a number of convincing numerical simulations have been provided, showing the relevance of that approach.
Feasibility of this methodology is nicely illustrated in \cite{AtwellKing} for heat equations and in \cite{Kunisch1} for the Burgers equation.
But, in these papers, the above theoretical issue has been let as an open problem. In this paper, we provide the first general theorem providing a positive answer.

\medskip

The paper is structured as follows.
In Section \ref{sec1.1} we give all assumptions under which our general result will be established.
We provide in Section \ref{secPOD} some elements on the POD approach. Our main result is stated in Section \ref{sec:main-result}. An idea of the strategy of its proof is given in Section \ref{sec:strategy}.
Section \ref{sec:someresults} contains some reminders and useful results on POD, useful in the proof of the main result.
Section \ref{sec:convergence} is devoted to proving the main theorem.
In Section \ref{sec:ccl}, we give a conclusion and some open problems and perspectives.
Finally, in Appendix \ref{app:riccati}, we establish an aymptotic result in Riccati theory, which is instrumental in the proof of our main result.

\subsection{General setting and assumptions}\label{sec1.1}
Let $H$ and $U$ be real Hilbert spaces. Let $A:D(A)\rightarrow H$ be a densely defined, closed selfadjoint operator, such that there exists some $\alpha\in\R$ for which $A-\alpha\,\mathrm{id}$ is dissipative. By the Lumer-Phillips theorem (see \cite{EngelNagel, Pazy}), $A$ generates a quasicontraction $C_0$ semi-group $(S(t))_{t\geq 0}$ on $H$, i.e., satisfying $\Vert S(t)\Vert\leq e^{\alpha t}$ for every $t\geq 0$.
Let $B\in L(U,H)$ be a bounded control operator.
Consider the control system
\begin{equation}\label{control_system}
\dot y(t) = A y(t) + Bu(t),\quad t\in(0,+\infty)
\end{equation}
with controls $u\in L^2(0,+\infty;U)$. % (see \cite{TucsnakWeiss} for this general context).

The objective of our paper is to exponentially stabilize the control system \eqref{control_system} with a feedback control designed from a finite-dimensional projection of \eqref{control_system} obtained by POD.

\medskip

In what follows, we denote by $\Vert \cdot\Vert$ the norm in $H$ and by $\langle\cdot,\cdot\rangle$ the corresponding scalar product.
Throughout the paper, we make the following assumptions.

\begin{enumerate}[label=$\bf (H_\arabic*)$]
\item\label{H1} We assume that the Hilbert space $H$ can be written as the direct orthogonal sum
$$
H=E_\ell \overset{\bot}{\oplus} F_\ell
$$
where $E_\ell\subset D(A)$ is of dimension $\ell$, $F_\ell$ is a closed subspace of $H$ such that $F_\ell\cap D(A)$ is dense in $F_\ell$ (for the induced topology), satisfying $E_\ell\bot F_\ell$ and
$$
AE_\ell\subset E_\ell\qquad\textrm{and}\qquad A\left( F_\ell\cap D(A)\right) \subset F_\ell
$$
(invariance under $A$).

We denote by $P_\ell$ the orthogonal projection of $H$ onto $E_\ell$; then $\mathrm{id}-P_\ell$ is the orthogonal projection of $H$ onto $F_\ell$.
By \ref{H1}, we have
\begin{equation}\label{spectralproperties}
%P_\ell A = A P_\ell = P_\ell A P_\ell,\qquad
P_\ell A (\mathrm{id}-P_\ell)_{|D(A)} = 0,\qquad (\mathrm{id}-P_\ell) A P_\ell = 0
\end{equation}
and
\begin{equation}\label{decompA}
A = P_\ell AP_\ell +  (\mathrm{id}-P_\ell)A (\mathrm{id}-P_\ell) .
\end{equation}
It follows from the Hille-Yosida theorem (see, e.g., \cite{EngelNagel, Pazy}) that
\begin{itemize}
\item the (bounded) operator $P_\ell AP_\ell$ on $E_\ell$ (which can be identified with a matrix of size $\ell\times\ell$) generates on $E_\ell$ the uniformly continuous semigroup $\left( P_\ell S(t) P_\ell \right)_{t\geq 0}$, with $P_\ell S(t) P_\ell = \exp(t P_\ell A P_\ell)$ for every $t\geq 0$;
\item the operator $(\mathrm{id}-P_\ell)A (\mathrm{id}-P_\ell)$ on $F_\ell$, of domain $F_\ell\cap D(A)$, generates the (quasicontraction) $C_0$ semigroup $\left( (\mathrm{id}-P_\ell) S(t) (\mathrm{id}-P_\ell) \right)_{t\geq 0}$.
\end{itemize}
We make the two following assumptions on those semigroups:

\item\label{H2}
We assume that the latter semigroup is exponentially stable, i.e., that there exists $\gamma>0$ such that
$$
\Vert (\mathrm{id}-P_\ell) S(t) (\mathrm{id}-P_\ell) z\Vert \leq e^{-\gamma t}\Vert(\mathrm{id}-P_\ell) z\Vert \qquad\forall t\geq 0\quad \forall z\in H.
$$
\item\label{H3} The operator $P_\ell AP_\ell$ (restriction of $A$ to $E_\ell$) can be identified with a selfadjoint $\ell\times\ell$ matrix, which is therefore diagonalizable with real-valued eigenvalues. We assume that all eigenvalues of $P_\ell AP_\ell$ are simple and have a positive real part. We define
\begin{equation}\label{def_betaell}
\beta_\ell  = \min \{\lambda\ \mid\ \lambda\in\mathrm{Spec}(P_\ell AP_\ell)\} >0 .
\end{equation}
%i.e., the minimal real part of eigenvalues of $P_\ell AP_\ell$ is positive.
%has $\kappa$ distinct complex eigenvalues, denoted by $\lambda_1,\ldots,\lambda_\kappa$ (each of them has algebraic multiplicity $\ell_j$ and we have $\ell_1+\ldots+\ell_\kappa=\ell$) and indexed such that $\Real(\lambda_1)\geq\Real(\lambda_2)\geq\dots\geq\Real(\lambda_\kappa)$.

In other words, we assume in particular that $0\notin \mathrm{Spec}(A)$, that $E_\ell$ is the finite-dimensional instable part of the system and that $F_\ell$ is the exponentially stable part.

\item\label{H4} We assume that the pair $(P_\ell AP_\ell,P_\ell B)$ satisfies the Kalman condition
$$
\mathrm{rank}( P_\ell B, P_\ell AP_\ell B, \ldots, P_\ell A^{\ell-1}P_\ell B)=\ell .
$$
This assumption is satisfied under the following much stronger assumption of unique continuation (which is equivalent, by duality, to approximate controllability for the system \eqref{control_system}): there exists $T>0$ such that, given any $z\in H$, if $B^*S(t)^*z=0$ for every $t\in[0,T]$ then $z=0$.
\end{enumerate}

The assumptions \ref{H1}, \ref{H2}, \ref{H3} and \ref{H4} are satisfied, for instance, for heat-like equations with internal control, i.e., when
$$
A=\triangle + a\,\mathrm{id}\quad\text{and}\quad B=\chi_\omega
$$
where $a\in L^\infty(\Omega,\R)$, $\triangle$ is the Dirichlet-Laplacian on a bounded $C^2$ domain $\Omega$ of $\R^d$, $\omega\subset\Omega$ is a nonempty open subset of $\Omega$ and $\chi_\omega$ is its characteristic function. %, and $a\in L^\infty(\Omega,\R)$ is such that $A$ does not have $0$ as an eigenvalue.
Taking $H=U=L^2(\Omega,\R)$, the operator $A=\triangle$ on $D(A)=H^1_0(\Omega,\R)\cap H^2(\Omega,\R)$ is selfadjoint and of compact inverse and thus is diagonalizable. We assume that $a$ and $\Omega$ are such that the spectrum of $A$ is simple (this is true under generic assumptions, see \cite{Teytel}) and such that $0$ is not an eigenvalue. Then there exists a Hilbert basis $(\phi_j)_{j\in\N^*}$ of $H$ consisting of real-valued eigenfunctions corresponding to the real eigenvalues
$$
-\infty\leftarrow\lambda_j\cdots<\lambda_{\ell+1}<0< \lambda_\ell<\lambda_{\ell-1}<\cdots <\lambda_1
$$
(with a slight abuse of notation because the number $\ell$ of instable modes may be equal to $0$). Taking $E_\ell=\mathrm{Span}\{\phi_j\ \mid\ j=1\ldots\ell\}$ and $F_\ell=\mathrm{Span}\{\phi_j\ \mid\ j\geq \ell+1\}$, Assumptions \ref{H1}, \ref{H2} and \ref{H3} are satisfied.
Assumption \ref{H4} is satisfied because of unique continuation: indeed we have $\chi_\omega\phi_j\neq 0$ for  $j=1,\dots,\ell$.

Of course, when $a$ is such that all eigenvalues of $A$ are negative, any solution of \eqref{control_system} converges exponentially to $0$.
We are interested in the case where there are (a finite number of) positive eigenvalues, i.e., $\ell>0$, and then stabilization is an issue.

\medskip

More generally, the assumptions \ref{H1}, \ref{H2}, \ref{H3} are satisfied when $A-\alpha\,\mathrm{id}$ is of compact inverse, with $A$ having a finite number of instable (positive) eigenvalues which are moreover simple. Our framework even allows for more general situations in which spectrum may not be discrete, but does not involve the case of wave-like equations for instance (for which $A$ is not selfadjoint). %Indeed, the orthogonal decomposition assumed in \ref{H1} is instrumental in our analysis. Besides, as said above,
Assumption \ref{H4} follows from unique continuation but is much weaker and may be satisfied for finite-rank control operators $B$.

\medskip

Thanks to the assumptions \ref{H1}, \ref{H2}, \ref{H3} and \ref{H4}, to stabilize \eqref{control_system} it would suffice to focus on the finite-dimensional instable part $E_\ell$ of the infinite-dimensional system \eqref{control_system}, as this was done for instance in \cite{Barbu, CoronTrelat1, CoronTrelat2} (see also \cite[page 711]{Russell} and \cite[Chapter 3]{WangXu}).
However, in practice eigenelements are not known in general or may be difficult to compute numerically. In particular, the integer $\ell$ is not known in general or may be difficult to compute although we know its existence.

Stabilizing the system from a finite-dimensional approximation of \eqref{control_system} that is \emph{not} of a spectral nature but which is anyway, in some sense, \emph{compatible} with the above spectral decomposition, is the main challenge that we address in this paper.

\medskip

We address this issue by approximating the control system \eqref{control_system} thanks to the POD method, described hereafter, which generates a $m$-dimensional \emph{reduced-order} control system, with $m$ sufficiently large ($m\geq\ell$ will be enough).

In what follows, we consider an arbitrary element
$$
y_0\in D(A)
$$
which, used as an initial condition, generates the trajectory $y(t)=S(t)y_0$, solution of \eqref{control_system} with $u=0$. We will consider it to generate snapshots in the POD method as explained next.

\subsection{Proper Orthogonal Decomposition (POD)}\label{secPOD}
The main idea of POD is to design an orthogonal basis of reduced order (called a POD basis) from a given collection of data (called snapshots). In order to face with too costly computations of a too complex model, the rationale behind POD is to generate a reduced set of basis functions able to capture the essential information of the physical process under consideration.
POD has been developed long time ago, and independently, by many authors in various contexts. POD is closely related to Karhunen-Lo\`eve decompositions and to principal component analysis (PCA) or factor analysis. It has been widely used in the context of fluid mechanics and in particular turbulence (see \cite{Berkooz, Chambers, Gunzburger, Holmes, KunischVolkwein2002, Lumley}) of chemical reactions (see \cite{Ly, Singer, Theod}) and it has become a classical approach for nonlinear model reduction (see \cite{Chapelle, Hesthaven, HinzeVolkwein, KunischVolkwein2001, Lassila, Markovsky, Quarteroni, Sirovich} and see \cite{AtwellKing, KunischVolkwein1999, Kunisch1, Kunisch2, Ravindran, TroltzschVolkwein} for applications to control of PDEs).
The POD method consists of designing an unstructured low-rank approximation of a matrix composed of snapshots of the state.
It can roughly be thought of as a Galerkin approximation in the spatial variable, built from values $y_k=y(t_k)$ of solutions of the physical system taken at prescribed times $0=t_1 <t_2 <\dots<t_n <+\infty$, for some $n\in\N^*$. These values $y_k$ (assumed to be known) are called \emph{snapshots}.

Here, we take $n$ snapshots
\begin{eqnarray}\label{12168}
y_k=y(kT;y_0,0)=S(kT)y_0,\quad k=1,\ldots, n
\end{eqnarray}
of the solution $y(\cdot;y_0,0)$ to \eqref{control_system} with initial condition $y_0$ and with the control $u=0$, taken at times $kT$, for some $T>0$.
We set
\begin{equation}\label{defDn}
D_n =\mathrm{Span} ( y_1,y_2,\dots, y_n )\qquad\textrm{and}\qquad d_n = \dim D_n.
\end{equation}
Note that, since $y_0\in D(A)$, we have $D_n\subset D(A)$.

%The subspace $D_n$ of $H$ inherits of the induced norm and scalar product.
%By regularization properties due to analyticity, we have $y_k\in H_1$ when $k\geq 1$, hence $D_n$ is a subspace of $H_1$ and thus inherits of the induced norm and scalar product.

Given some integer $m\leq d_n$, the POD method consists of determining a subspace $D_{n,m}$ of $D_n$, of dimension $\leq m$, such that the mean square discrepancy between all snapshots $y_k$ and their orthogonal projection $\Pi_{D_{n,m}}y_n$ onto $D_{n,m}$ is minimal, i.e., it consists of minimizing the functional
\begin{equation}\label{JDnm}
\boxed{
J(D_{n,m}) = \sum_{k=1}^n \Big\Vert  y_k-\Pi_{D_{n,m}}y_k\Big\Vert ^2 =\sum_{k=1}^n \Big\Vert  \Pi^{D_n}_{D_{n,m}^\perp}y_k\Big\Vert ^2
}
\end{equation}
over all possible subspaces $D_{n,m}$ of $D_n$ of dimension $\leq m$ (equivalently, of dimension equal to $m$). % when $d_n=\dim(D_n)\geq m$).
Here, $\Pi^{D_n}_{D_{n,m}^\perp}$ is the orthogonal projection onto the orthogonal $D_{n,m}^\perp$ of $D_{n,m}$ in $D_n$.
This minimization problem has at least one solution $\overline D_{n,m}$ (see, e.g., \cite{Volkwein}) and we denote by $\overline J_{n,m}$ the optimal value, but the optimal approximating subspace $\overline D_{n,m}$ may not be unique.

The problem is often formulated as follows. % when $d_n=\dim(D_n)\geq m$.
Assume that $D_{n,m}=\mathrm{Span}( \psi_1,\ldots,\psi_m)$ and complete these $m$ orthonormal vectors into an orthonormal basis $(\psi_j)_{j=1,\ldots,d_n}$ of $D_n$; write $y_k=\sum_{j=1}^{d_n} \langle y_k,\psi_j\rangle \psi_j$ and $\Pi_{D_{n,m}}y_k = \sum_{j=1}^m \left\langle y_k,\psi_j \right\rangle  \psi_j$. Then, the POD method consists of minimizing
$$
\sum_{k=1}^n \Big\Vert  y_k-\sum_{j=1}^m \langle y_k,\psi_j\rangle \psi_j \Big\Vert^2
$$
over all possible orthonormal families $(\psi_j)_{j=1,\ldots,m}$ in $D_n$.
A subspace $\overline D_{n,m}=\mathrm{Span}( \overline \psi_1,\ldots,\overline \psi_m)$, optimal solution of the minimization problem \eqref{JDnm}, i.e., such that $J(\overline D_{n,m})=\overline J_{n,m}$, is then used as a best approximating subspace of $D_n$ of dimension $m$. Any orthonormal basis $(\psi_j)_{j=1,\ldots,m}$ of $\overline D_{n,m}$ is called a POD basis of rank $m$.

Other properties of POD, related to SVD (Singular Value Decomposition), are recalled further in Section \ref{sec:SVD}.

\subsection{Main result}\label{sec:main-result}
\paragraph{POD reduced-order control system.} %Finite-dimensional projection of the control system \eqref{control_system}.}
Keeping the assumptions and notations of  the previous subsections, we fix an arbitrary $T>0$ and we consider an optimal solution $\overline D_{n,m}$ of the (POD) minimization problem \eqref{JDnm}, i.e., a best approximating $m$-dimensional subspace of  the space $D_n\subset D(A)$ defined by \eqref{defDn}.

Applying the orthogonal projection $\Pi_{\overline D_{n,m}}$ to the control system \eqref{control_system} yields
$$
\frac{d}{dt} \Pi_{\overline D_{n,m}} y(t) = \Pi_{\overline D_{n,m}} A \Pi_{\overline D_{n,m}} y(t) + \Pi_{\overline D_{n,m}} B u(t) + \Pi_{\overline D_{n,m}} A (\mathrm{id}-\Pi_{\overline D_{n,m}} ) y(t) .
$$
The last term at the right-hand side of the above equation is seen as a perturbation term, and we are thus led to consider the following \emph{POD reduced-order control system} in the space $\overline D_{n,m}$
\begin{equation}\label{reduced_control_system}
\boxed{ \dot Y(t) = A_{n,m} Y(t) + B_{n,m} v(t) }
\end{equation}
with
\begin{equation*}%\label{coefficient}
A_{n,m} = \Pi_{\overline D_{n,m}} A \Pi_{\overline D_{n,m}} \quad\textrm{and}\quad B_{n,m} = \Pi_{\overline D_{n,m}} B .
\end{equation*}
Note that $A_{n,m} $ is well defined because $\overline D_{n,m}\subset D(A)$. The control system \eqref{reduced_control_system} is a linear autonomous control system in the $m$-dimensional
space $\overline D_{n,m}$ with controls $v\in L^2(0,+\infty;U)$.
The operator $A_{n,m}$ can be identified with a square matrix of size $m$ and $B_{n,m}$ with a matrix of size $m\times\dim(U)$ (with $U$ of finite or infinite dimension).

\paragraph{Stabilizing the reduced-order control system.}
In the proof of our main result (Theorem \ref{mainthm} hereafter), we will prove that the POD reduced-order $m$-dimensional control system \eqref{reduced_control_system} is stabilizable when $m\geq\ell$ and $n$ is large enough. To design an exponentially stabilizing linear feedback $v=K_{n,m}Y$, we use the Riccati theory (see also Appendix \ref{app:riccati} for some reminders on the Riccati theory).

Given any $\varepsilon\geq 0$, since the pair $(A_{n,m},B_{n,m})$ is stabilizable, by \cite[Corollary 2.3.7 page 55]{Abou-Kandil} there exists a (unique) maximal symmetric positive semidefinite solution $P_{n,m}(\varepsilon)$ (of size $m\times m$) of the algebraic Riccati equation
\begin{equation}\label{RiccatiPnm}
A_{n,m}^* P_{n,m}(\varepsilon) + P_{n,m}(\varepsilon) A_{n,m} - P_{n,m}(\varepsilon) B_{n,m} B_{n,m}^* P_{n,m}(\varepsilon) + \varepsilon I_{m} = 0 .
\end{equation}
Here and throughout, $I_k$ denotes the identity matrix of size $k$.
Moreover, $A_{n,m} - B_{n,m} B_{n,m}^* P_{n,m}(\varepsilon)$ is semi-stabilizing, i.e., its eigenvalues have nonpositive real part. If $\varepsilon>0$ then $P_{n,m}(\varepsilon)$ is positive definite, and $A_{n,m} - B_{n,m} B_{n,m}^* P_{n,m}(\varepsilon)$ is Hurwitz, i.e., its eigenvalues have a negative real part (these facts are established in Appendix \ref{app:riccati}).
Therefore, setting
$$
K_{n,m}(\varepsilon) = -B_{n,m}^* P_{n,m}(\varepsilon) ,
$$
the linear feedback
$$
v = K_{n,m}(\varepsilon) Y = -B_{n,m}^* P_{n,m}(\varepsilon) Y
$$
exponentially stabilizes the reduced-order control system \eqref{reduced_control_system} to the origin.

\begin{remark}
Given any $\varepsilon\geq 0$ and any $Y_0\in\overline D_{n,m}$, there exists a unique optimal control minimizing the functional
\begin{equation*}%\label{cost_function}
\int_0^{+\infty} \left( \varepsilon \Vert Y(t)\Vert_{\overline D_{n,m}}^2 + \Vert v(t)\Vert_U^2 \right) dt
\end{equation*}
over all possible controls $v\in L^2(0,+\infty; U)$, where $Y(\cdot)$ is the solution to \eqref{reduced_control_system} with control $v$ and with initial condition $Y(0)=Y_0$. If $\varepsilon>0$ then the optimal control is exactly the stabilizing feedback $v = K_{n,m}(\varepsilon) Y$ .
\end{remark}

We will prove that the closed-loop matrix $A_{n,m} + B_{n,m} K_{n,m}(\varepsilon)$, which is Hurwitz if $\varepsilon>0$, actually remains \emph{uniformly} Hurwitz as $\varepsilon\rightarrow 0$ (precise asymptotic results are established in Appendix \ref{app:riccati}). In particular, the matrix $\exp(t(A_{n,m} + B_{n,m} K_{n,m}(\varepsilon)))$ decreases exponentially, with an exponential rate which remains uniformly bounded below by some positive constant as $\varepsilon\rightarrow 0$. %, i.e., there exist $C_1>0$ and $C_2>0$ such that $\Vert\exp(t(A_{n,m} + B_{n,m} K_{n,m}(\varepsilon)))\Vert\leq C_1e^{-C_2t}$ for every $\varepsilon$ small enough.

\paragraph{Main result.}
We now use the above feedback matrix $K_{n,m}(\varepsilon)$ in the original infinite-dimensional control system \eqref{control_system}, by taking the feedback control
\begin{equation}\label{finitedimfeedback}
u = K_{n,m}(\varepsilon)\Pi_{\overline D_{n,m}} y = -B_{n,m}^* P_{n,m}(\varepsilon)\Pi_{\overline D_{n,m}} y .
\end{equation}
Since $B$ is bounded, the operator $A + B K_{n,m}(\varepsilon)\Pi_{\overline D_{n,m}}$ is defined on $D(A)$ and generates a $C_0$ semigroup.
Our main result establishes that, under appropriate assumptions, this semigroup is exponentially stable. In other words, the ``finite-dimensional" feedback \eqref{finitedimfeedback}, which exponentially stabilizes the finite-dimensional control system \eqref{reduced_control_system}, also exponentially stabilizes the infinite-dimensional control system \eqref{control_system} if the number $n$ of snapshots is large enough and if $\varepsilon>0$ is small enough.

\begin{theorem}\label{mainthm}
We make the assumptions \ref{H1}, \ref{H2}, \ref{H3} and \ref{H4}, and we assume that the pair $(P_\ell AP_\ell,P_\ell y_0)$ satisfies the Kalman condition, i.e.,
\begin{equation}\label{Kalmany0}
\mathrm{rank}(P_\ell y_0,P_\ell AP_\ell y_0,\ldots,P_\ell A^{\ell-1}P_\ell y_0) = \ell .
\end{equation}
Let $m\geq\ell$ be arbitrary.
There exist $\varepsilon_0>0$ and $n_0\in\N$ such that, for every $\varepsilon\in[0,\varepsilon_0]$ and every $n\geq n_0$, the control system \eqref{control_system} in closed-loop with the feedback $u=K_{n,m}(\varepsilon)\Pi_{\overline D_{n,m}} y$,
\begin{equation}\label{stateclosedloop}
\boxed{ \dot y(t) = \left( A + B K_{n,m}(\varepsilon)\Pi_{\overline D_{n,m}} \right) y(t) }
\end{equation}
is exponentially stable, meaning that any solution of \eqref{stateclosedloop} converges exponentially to $0$ in $H$ as $t\rightarrow+\infty$.
\end{theorem}

\begin{remark}[On Assumption \eqref{Kalmany0}]\label{rem1}
Under Assumption \eqref{Kalmany0}, we have that $d_n=\dim(D_n)\geq\ell$ when $n\geq \ell$. This is why we can take $m\geq\ell$ in the theorem.
Since $P_\ell AP_\ell$ is selfadjoint, recalling that $y_k=S(kT)y_0=S(T)^ky_0$ and noting that $P_\ell S(T)P_\ell = \exp(T P_\ell AP_\ell)$, we see that Assumption \eqref{Kalmany0} is equivalent to the assumption that the pair $(P_\ell S(T)P_\ell,P_\ell y_0)$ satisfies the Kalman condition, i.e.,
\begin{equation}\label{kalmanexp}
\mathrm{rank}(P_\ell y_0, P_\ell y_1, \ldots, P_\ell y_{\ell-1}) = \mathrm{rank}(P_\ell y_0,P_\ell S(T)P_\ell y_0,\ldots,P_\ell S(T)^{\ell-1}P_\ell y_0)=\ell .
\end{equation}
This is rather this condition that we will use in the proof.

A second remark is the following.
Let $(\phi_{1},\ldots,\phi_{ \ell})$ be an orthonormal basis of $E_\ell$, consisting of eigenvectors of $P_\ell AP_\ell$, corresponding to the (real-valued) eigenvalues $\lambda_j$, $j=1,\ldots, \ell$.
Assumption \eqref{Kalmany0} (equivalently, Assumption \eqref{kalmanexp}) is satisfied if and only if all eigenvalues of $P_\ell AP_\ell$ are simple and
\begin{equation}\label{generic_condition}
\langle \phi_{j},\, P_\ell y_0\rangle \neq 0  \qquad\forall j\in\{1,\ldots, \ell\}
\end{equation}
i.e., the component of $P_\ell y_0$ in the direction $\phi_{j}$ is nonzero, for every $j\in\{1,\ldots, \ell\}$.
The condition \eqref{generic_condition} is generic in the sense that the set of $y_0\in D(A)$ of which one of the $\ell$ first spectral modes is zero has codimension $1$ (and thus has measure zero) in $D(A)$.
\end{remark}

\begin{remark}\label{rem_exprate}
Define the best exponential decay rate $\gamma^*$ of an exponentially stable $C_0$ semigroup $(T(t))_{t\geq 0}$ on $H$ as the supremum of all possible $\gamma>0$ for which there exists $M\geq 1$ such that $\Vert T(t)\Vert_{L(H)} \leq M\mathrm{e}^{-\gamma t}$ for every $t\geq 0$, i.e., $\gamma^* = -\inf_{t>0} \frac{1}{t} \ln \Vert T(t)\Vert_{L(H)} = -\lim_{t\rightarrow+\infty} \frac{1}{t} \ln \Vert T(t)\Vert_{L(H)}$ (see \cite{EngelNagel, Pazy}).

Let $\gamma>0$ be the best decay rate of the exponentially stable quasicontraction $C_0$ semigroup $\left( (\mathrm{id}-P_\ell) S(t) (\mathrm{id}-P_\ell) \right)_{t\geq 0}$ (see Assumption \ref{H2}). Let $\gamma_\varepsilon>0$ be the best decay rate of the matrix $\exp(t(A_{n,m} + B_{n,m} K_{n,m}(\varepsilon)))$ ($-\gamma_\varepsilon$ is the spectral abscissa of $A_{n,m} + B_{n,m} K_{n,m}(\varepsilon)$).

Then, in Theorem \ref{mainthm}, the growth bound $\gamma^*(\varepsilon)$ of the exponentially stable $C_0$ semigroup generated by $A + B K_{n,m}(\varepsilon)\Pi_{\overline D_{n,m}}$ satisfies
\begin{equation}\label{decayrate}
\lim_{\varepsilon\rightarrow 0} \gamma^*(\varepsilon) = \min(\gamma,\gamma_\varepsilon) .
\end{equation}
\end{remark}

\subsection{Strategy of the proof}\label{sec:strategy}
Establishing Theorem \ref{mainthm} is easier under the additional assumption
\begin{equation}\label{subspace}
E_\ell\subset \overline D_{n,m}
\end{equation}
and we first sketch the argument under this simplifying assumption.
In this case, we write
\begin{equation}\label{1.18}
\overline D_{n,m}=E_\ell\overset{\bot}{\oplus} F_1
\end{equation}
where $F_1$ is a subspace of $F_\ell$.
Since $\mathrm{Ran}(AP_\ell)\subset E_\ell\,\bot\, F_1$, we have $P_{F_1} AP_\ell=0$ where $P_{F_1}$ is the orthogonal projection onto $F_1$. In the decomposition \eqref{1.18}, the control system \eqref{reduced_control_system} is written as
\begin{eqnarray}
\dot Y_1 &=& P_\ell AP_\ell Y_1 + P_\ell B u \label{1.19} \\
\dot Y_2 &=& P_{F_1}AP_{F_1} Y_2 + P_{F_1}Bu \label{1.19b}
\end{eqnarray}
By \ref{H4}, the pair $(P_\ell AP_\ell, P_\ell B)$ satisfies the Kalman condition and thus the subsystem \eqref{1.19} is stabilizable.
Besides, by \ref{H2}, the subsystem \eqref{1.19b} is exponentially stable with control $u=0$.
It follows from Appendix \ref{app:riccati} that the control system \eqref{1.19}-\eqref{1.19b} (which is equivalent to \eqref{1.18}) is stabilizable by the Riccati procedure: the optimal feedback $u = K_{n,m}(\varepsilon)P_\ell Y_1 + K_{n,m}(\varepsilon)P_{F_1} Y_2$ exponentially stabilizes the system \eqref{1.19}-\eqref{1.19b}, i.e., the closed-loop matrix
\begin{equation}\label{1.20}
\begin{pmatrix}
P_\ell (A+B K_{n,m}(\varepsilon))P_\ell & P_\ell BK_{n,m}(\varepsilon)P_{F_1}  \\
P_{F_1}BK_{n,m}(\varepsilon)P_\ell & P_{F_1}(A+B K_{n,m}(\varepsilon))P_{F_1} \\
\end{pmatrix}
\end{equation}
is Hurwitz, for every $\varepsilon>0$.
Moreover, by \eqref{a.55} in Appendix \ref{app:riccati}, we have $K_{n,m}(\varepsilon)P_{F_1}\rightarrow 0$ as $\varepsilon\rightarrow 0$, and the closed-loop matrix \eqref{1.20} remains uniformly Hurwitz as $\varepsilon\rightarrow 0$. This fact is important in our analysis.

Now, plugging this finite-dimensional feedback into the initial control system \eqref{control_system}, we obtain the closed-loop system \eqref{stateclosedloop}, with $A + B K_{n,m}(\varepsilon)\Pi_{\overline D_{n,m}}$ that is written in the above decomposition as the infinite-dimensional matrix
\begin{equation}\label{1.22}
\left(\begin{array}{c|cc}
P_\ell (A+B K_{n,m}(\varepsilon))P_\ell &P_\ell B K_{n,m}(\varepsilon)P_{F_1} &0
\\[1mm] \hline
\\[-3mm]
P_{F_1}B K_{n,m}(\varepsilon)P_\ell & P_{F_1}(A+B K_{n,m}(\varepsilon))P_{F_1} &P_{F_1}A(\mathrm{id}-P_\ell-P_{F_1})
\\[1mm]
0&(\mathrm{id}-P_\ell-P_{F_1})A P_{F_1}&(\mathrm{id}-P_\ell-P_{F_1})A(\mathrm{id}-P_\ell-P_{F_1})
\end{array}\right).
\end{equation}
Since $K_{n,m}(\varepsilon)P_{F_1}\rightarrow 0$ as $\varepsilon\rightarrow 0$, the matrix \eqref{1.22} is approximately lower block triangular, with the first diagonal block being exponentially stable (because \eqref{1.20} is Hurwitz) and the second diagonal block being exponentially stable as well (because it is close to $(\mathrm{id}-P_\ell)A (\mathrm{id}-P_\ell)$ as $\varepsilon$ is small enough).
Therefore, \eqref{1.22} is exponentially stable and Theorem \ref{mainthm} follows, under the simplifying assumption \eqref{subspace}.

\medskip

In general, however, \eqref{subspace} is not true: there is indeed no reason that, when performing the POD reduction, the space $\overline D_{n,m}$ contain the spectral subspace $E_\ell$. Indeed, \emph{``POD does not see eigenmodes"}.

Anyway, our complete analysis, done in Section \ref{sec:convergence}, will reveal that this is \emph{almost} the case: we will prove in particular that
$$
\overline D_{n,\ell} \simeq E_\ell
$$
when $n$ is large enough, which implies that the inclusion \eqref{subspace} is almost satisfied (because $\overline D_{n,\ell}\subset \overline D_{n,m}$). Establishing such a result will require a quite fine analysis. This shows that, in some sense, our problem is a small perturbation problem of \eqref{1.22}
when $n$ is large. Theorem \ref{mainthm} will then be proved in Section \ref{sec:proofs}, using an asymptotic result in Riccati theory developed in Appendix \ref{app:riccati}, roughly stating that, when considering a linear system having an instable part and a stable part, the Riccati stabilization procedure with weight $\varepsilon$ on the state and weight $1$ on the control yields feedbacks that essentially act on the instable part for $\varepsilon$ small. The results established in Appendix \ref{app:riccati} are a bit delicate and require a particular care, although all notions thereof remain quite elementary.

\section{Some results on POD}\label{sec:someresults}
\subsection{Relationship with Singular Value Decomposition (SVD)}\label{sec:SVD}
It is well known that optimal solutions of POD can be expressed thanks to SVD.
%Completing the orthonormal basis $(\phi_1,\ldots,\phi_\ell)$ of $E_\ell$ to a Hilbert basis $(\phi_j)_{j\in\N^*}$ of $H=E_\ell \oplus F_\ell$,
 Let $\{y_1,\dots,y_n\}$ be given by \eqref{12168}.
We consider the $(\infty\times n)$-matrix
$$
Y_n = \left( y_1 , \ldots , y_n \right)
$$
expressed in an arbitrary Hilbert basis of $H$.
We have $D_n=\mathrm{Ran}(Y_n)$ and $d_n=\mathrm{rank}(Y_n)=\dim(D_n)$.
Since the matrix $Y_n$ is of finite rank $d_n$, SVD works exactly as in finite dimension (because $Y_nY_n^*$ and $Y_n^*Y_n$ are compact and selfadjoint, see \cite{Beutler}). According to the SVD theorem, we have
$$
Y_n = V_n \Sigma_n U_n^*
$$
where $V_n$ is an orthogonal matrix of infinite size (unitary operator in $H$, consisting of eigenvectors of $Y_nY_n^*$), $U_n$ is an orthogonal matrix of size $n$ (consisting of eigenvectors of $Y_n^*Y_n$) and $\Sigma_n$ is a matrix of size $\infty\times n$ consisting of the diagonal
 $\sigma_{n,1}\geq \sigma_{n,2}\geq \cdots\geq\sigma_{n,n}\geq 0$
 (singular values of $Y_n$), completed with zeros. The singular values of $Y_n$ are nonnegative real numbers, with the $d_n$ first ones being positive and all others being zero.
Denoting by $u_{n,i}\in H$ and $v_{n,i}\in H$ the columns of $U_n$ and $V_n$, we have
\begin{equation}\label{singular_values}
Y_n = \sum_{i=1}^{d_n} \sigma_{n,i} v_{n,i} u_{n,i}^* .
\end{equation}
Let $m\leq d_n$ be an integer.
We define the $(\infty\times m)$-matrix $V_{n,m}$ as the submatrix of $V_n$ consisting of the first $m$ columns of $V_n$, which are $v_{n,1},\ldots,v_{n,m}$. Similarly, we define the $(n\times m)$-matrix $U_{n,m}$ as the submatrix of $U_n$ consisting of the first $m$ columns of $U_n$, which are $u_{n,1},\ldots,u_{n,m}$. Finally, we define the square diagonal matrix $\Sigma_{n,m}$ of size $m$, consisting of the elements $\sigma_{n,1}\geq\cdots\geq\sigma_{n,m}$.
It is then well known (see, e.g., \cite{Golub, Markovsky}) that the ``best" projection of rank $\leq m$ onto $D_n=\mathrm{Ran}(Y_n)$ is $V_{n,m} V_{n,m}^*$ and that the ``best" approximation of rank $\leq m$ of the matrix $Y_n$, over all matrices of rank $\leq m$, is the matrix
$$
Y_{n,m} = V_{n,m} V_{n,m}^* Y_n = V_{n,m}\Sigma_{n,m} U_{n,m}^* = \sum_{i=1}^{m} \sigma_{n,i}v_{n,i} u_{n,i}^*.
$$
``Best" is understood here in the sense of the Frobenius norm as well as of the subordinate $2$-norm (and actually, of any norm invariant under the orthogonal group), and the Frobenius norm of $Y_n-Y_{n,m}$ is
$$
\Vert Y_n-Y_{n,m}\Vert_F^2 = \sum_{i=m+1}^{d_n} \sigma_{n,i}^2.
$$
Recall that the square $\Vert D\Vert_F^2$ of the Frobenius norm of a matrix $M$ (of any size, possibly infinite) is equal to the sum of squares of all elements of $M$. %, which is the trace of $MM^*$.
Moreover, when considering the Frobenius norm (also called Hilbert-Schmidt norm), we have uniqueness of the minimizer $Y_{n,m}$ if and only if $\sigma_{n,m}\neq \sigma_{n,m+1}$.
Note that the range of $Y_{n,m}$ is contained in the range of $Y_n$,  and thus $Y_{n,m}$ is also the best approximation of rank $\leq m$ of $Y_n$ over all possible matrices of rank $\leq m$ whose range is contained in the range of $Y_n$.

By definition, the quantity $J(D_{n,m})$ defined by \eqref{JDnm} is exactly the Frobenius norm of $Y_n-\Pi_{D_{n,m}}Y_n$:
$$
\boxed{
J(D_{n,m}) = \Vert Y_n-\Pi_{D_{n,m}}Y_n \Vert_F^2
}
$$
By the above remarks, since $\mathrm{rank}(\Pi_{D_{n,m}}Y_n)\leq m$, the POD problem is exactly equivalent to searching the best approximation of rank $\leq m$ of the matrix $Y_n$ for the Frobenius norm.
Therefore, we have
\begin{equation}\label{optvalue}
\boxed{
\overline J_{n,m} = J(\overline D_{n,m}) = \Vert Y_n-Y_{n,m}\Vert_F^2 = \sum_{i=m+1}^{d_n} \sigma_{n,i}^2
}
\end{equation}
with
\begin{equation*}
\begin{split}
\overline D_{n,m} &= \mathrm{Ran}(\Pi_{\overline D_{n,m}}) = \mathrm{Span}(v_{n,1},\ldots,v_{n,m}) ,\\
\Pi_{\overline D_{n,m}} &= V_{n,m} V_{n,m}^*, \\
Y_{n,m} &= \Pi_{\overline D_{n,m}} Y_n = V_{n,m}\Sigma_{n,m} U_{n,m}^* .
\end{split}
\end{equation*}

\subsection{Boundedness of the optimal value}
Recall that $y_k=S(kT) y_0$ for every $k\in\N$, where $T>0$ is fixed.
\begin{lemma}\label{lem_optvalue}
Under Assumptions \ref{H1} and \ref{H2}:
\begin{itemize}
\item There exists $C>0$ such that
\begin{equation}\label{ineq_sec_proof_lem_optvalue}
\Vert (\mathrm{id}-P_\ell)y_k\Vert + \Vert A(\mathrm{id}-P_\ell)y_k\Vert \leq C e^{-\gamma kT} \qquad\forall k\in\N.
\end{equation}
\item Given any $m\geq\ell$, the optimal value $\overline J_{n,m}$ (given by \eqref{optvalue}) of the minimization problem \eqref{JDnm} remains bounded as $n\rightarrow+\infty$.
\end{itemize}
\end{lemma}

\begin{proof}
Since $P_\ell$ commutes with $S(t)$ by Assumption \ref{H1}, we have $P_\ell y_k = P_\ell S(kT) y_0 = S(kT) P_\ell y_0$, and hence
$$
\Vert y_k-P_\ell y_k\Vert = \Vert (\mathrm{id}-P_\ell)S(kT)(\mathrm{id}-P_\ell) y_0\Vert \leq  e^{-\gamma kT} \Vert (\mathrm{id}-P_\ell) y_0\Vert
$$
where we have used Assumption \ref{H2} to get the latter inequality, and similarly,
$$
\Vert A y_k-AP_\ell y_k\Vert = \Vert (\mathrm{id}-P_\ell)S(kT)(\mathrm{id}-P_\ell) Ay_0\Vert \leq  e^{-\gamma kT} \Vert (\mathrm{id}-P_\ell) Ay_0\Vert
$$
and the first item follows because $y_0\in D(A)$.

For the second item, using the SVD interpretation of the POD, we have in particular (since $m\geq\ell$)
$$
\overline J_{n,m} \leq \Vert Y_n-P_\ell Y_n\Vert_F^2 = \sum_{k=1}^n \Vert y_k-P_\ell y_k\Vert^2.
$$
Therefore
$$
\overline J_{n,m} \leq  \Vert y_0\Vert^2 \sum_{k=1}^n e^{-2\gamma kT} \leq  \Vert y_0\Vert^2 \frac{ e^{-2\gamma T} }{ 1-e^{-2\gamma T} }
$$
and the lemma follows.
\end{proof}

\begin{remark}\label{rem_optvalue}
It follows from the second item of Lemma \ref{lem_optvalue} that there exists $C_1>0$ (which is the bound on $\overline J_{n,m}$) such that, for all integers $n$ and $m$ satisfying $n\geq m\geq\ell$, we have
$$
\sum_{k=1}^n \Vert (\mathrm{id}-\Pi_{\overline D_{n,m}})y_k\Vert^2 \leq C_1
$$
\end{remark}

\section{Proof of Theorem \ref{mainthm}}\label{sec:convergence}
\subsection{Several convergence results}
The lemmas established in this subsection are the key results to prove Theorem \ref{mainthm} in the next subsection.
Throughout, we make the assumptions \ref{H1}, \ref{H2}, \ref{H3}, \ref{H4} and \eqref{Kalmany0} (equivalently, \eqref{kalmanexp}).

\begin{lemma}\label{lem_CVproj}
If $m\geq\ell$ then there exists $C>0$ such that
%$\Vert  \Pi_{\overline D_{n,m}}\phi_j-\phi_j\Vert \rightarrow 0$ as $n\rightarrow+\infty$, for every $j\in\{1,2,\dots,\ell\}$, i.e., in other words,
\begin{equation*}%\label{operator1}
\Vert ( \mathrm{id} -  \Pi_{\overline D_{n,m}} ) P_\ell \Vert \leq C  e^{-n\beta_\ell T}\qquad  \forall n\geq m.
\end{equation*}
\end{lemma}

\begin{proof}
By Remark \ref{rem_optvalue}, there exists $C_1>0$ such that, for every $n\geq m$,
\begin{equation}\label{2.16}
\sum_{k=n-\ell+1}^n  \left\Vert y_k-\Pi_{\overline D_{n,m}}y_k\right\Vert^2\leq C_1 .
\end{equation}
Now, given any $k\in \{n-\ell+1, n-\ell+2,\dots, n\}$, we have
\begin{equation*}
\begin{split}
\left\Vert y_k- \Pi_{\overline D_{n,m}}y_k\right\Vert
&=\left\Vert (\mathrm {id}- \Pi_{\overline D_{n,m}}) P_\ell y_k +(\mathrm {id}- \Pi_{\overline D_{n,m}})(\mathrm {id} -P_\ell) y_k\right\Vert \\
& \geq
\left\Vert  (\mathrm {id}- \Pi_{\overline D_{n,m}}) P_\ell y_k\right\Vert - \left\Vert(\mathrm {id} - P_\ell)y_k\right\Vert
\end{split}
\end{equation*}
because $\Vert \mathrm {id}- \Pi_{\overline D_{n,m}}\Vert\leq 1$.
We infer that
\begin{equation}\label{2.16-1}
\left\Vert (\mathrm {id}- \Pi_{\overline D_{n,m}}) P_\ell y_k\right\Vert^2
\\
\leq  2\left\Vert y_k- \Pi_{\overline D_{n,m}}y_k\right\Vert^2 +2\left\Vert ( \mathrm {id}-P_\ell)y_k  \right\Vert^2.
\end{equation}
It follows from \eqref{ineq_sec_proof_lem_optvalue}, \eqref{2.16} and \eqref{2.16-1} that
\begin{equation}\label{2.17}
\sum^n_{k=n-\ell+1}  \left\Vert (\mathrm {id}- \Pi_{\overline D_{n,m}}) P_\ell y_k\right\Vert^2\leq C_2
\end{equation}
for some $C_2>0$ independent of $n\geq m$.
Now, for $k\in\{n-\ell+1,\ldots,n\}$, setting $j=k-n+\ell-1\in\{0,\ldots,\ell-1\}$, we have
$$
P_\ell y_k = P_\ell S(T)^{n-\ell+1}P_\ell \ P_\ell  S(T)^j P_\ell y_0
$$
and by Assumption \eqref{kalmanexp}, the elements $P_\ell  S(T)^j P_\ell y_0$, $j=0,\ldots,\ell-1$, generate the $\ell$-dimensional subspace $E_\ell$. We then infer from \eqref{2.17} that the norm of the operator $(\mathrm {id}- \Pi_{\overline D_{n,m}}) P_\ell S(T)^{n-\ell+1}P_\ell$ is bounded by a constant $C_3$ that is independent of $n\geq\ell$.
Since $P_\ell S(T)^{n-\ell+1}P_\ell = \exp( (n-\ell+1)T P_\ell AP_\ell)$ is boundedly invertible, it follows that
$$
\Vert (\mathrm {id}- \Pi_{\overline D_{n,m}})P_\ell \Vert \leq C_3 \Vert \exp( -(n-\ell+1)T P_\ell AP_\ell)\Vert
$$
The result follows by using the definition \eqref{def_betaell} of $\beta_\ell$ and the estimate of a matrix written in its canonical Jordan form.
\end{proof}

\begin{remark}
Thanks to Lemma \ref{lem_CVproj}, we have obtained that $\Pi_{\overline D_{n,m}} P_\ell \rightarrow P_\ell$ as $n\rightarrow+\infty$ (with an exponential convergence rate).
\end{remark}

\begin{lemma}\label{lem17}
Recalling that $\sigma_{n,i}$, $i\in\{1,\ldots,n\}$ are the singular values of $Y_n$ (see \eqref{singular_values}), we have
\begin{equation}\label{17}
\lim_{n\rightarrow+\infty}\sigma_{n,i}=+\infty\qquad\forall i\in\{1,\dots,\ell\}
\end{equation}
while all other singular values remain bounded.
\end{lemma}

\begin{proof}
We already know that the singular values $\sigma_{n,i}$, $i\in\{\ell+1,\ldots,n\}$ remain bounded as $n\rightarrow +\infty$ (this is because $\sum_{i=\ell+1}^n \sigma_{n,i}^2$ is bounded as $n\rightarrow +\infty$).

Let us prove that $\sigma_{n,\ell}\rightarrow+\infty$ as $n\rightarrow +\infty$. This will imply the result for $\sigma_{n,i}$ with $i\leq\ell$ because $\sigma_{n,1}\geq\sigma_{n,2}\geq\cdots\geq\sigma_{n,\ell}$.
By contradiction, let us assume that $\sigma_{n,\ell}$ remains bounded as $n\rightarrow +\infty$. By \eqref{optvalue}, we have $\overline J_{n,\ell-1}=\sigma_{n,\ell}^2+\overline J_{n,\ell}$ and hence, using Lemma \ref{lem_optvalue}, $\overline J_{n,\ell-1}$ remains bounded as well as $n\rightarrow +\infty$. But then we can repeat the reasoning done in Remark \ref{rem_optvalue} and then in Lemma \ref{lem_CVproj}, replacing $m$ with $\ell-1$: we thus obtain that $\Pi_{\overline D_{n,\ell-1}} P_\ell \rightarrow P_\ell$ as $n\rightarrow+\infty$. Hence $\mathrm{rank}(P_\ell)\leq\liminf\mathrm{rank}(\Pi_{\overline D_{n,\ell-1}} P_\ell)$, and this raises a contradiction because $\mathrm{rank}(P_\ell)=\ell$ while $\mathrm{rank}(\Pi_{\overline D_{n,\ell-1}} P_\ell) \leq \mathrm{rank}(\Pi_{\overline D_{n,\ell-1}}) \leq \dim\overline D_{n,\ell-1}\leq\ell-1$.
\end{proof}

\begin{remark}
Since the argument is by contradiction in the above lemma, we do not know the blow-up rate of $\sigma_{n,i}$, $i\in\{1,\ldots,\ell\}$.
\end{remark}

We next consider an appropriate decomposition of the vector space $\overline D_{n,m}$. In what follows we assume that $\ell<m\leq n$ (when $m=\ell$ we have $G_{n,m,\ell} = \{0\}$ below). We already know that $D_n = \mathrm{Ran}(V_{n,n}V_{n,n}^*) = \mathrm{Span}(v_{n,1},\ldots,v_{n,n})$ and we consider the decomposition
\begin{equation}\label{decompDnm}
\overline D_{n,m} = \mathrm{Ran}(V_{n,m}V_{n,m}^*) = \mathrm{Span}(v_{n,1},\ldots,v_{n,m}) = \overline D_{n,\ell} \overset{\bot}{\oplus} G_{n,m,\ell}
\end{equation}
with
$$
\overline D_{n,\ell} = \mathrm{Span}(v_{n,1},\ldots,v_{n,\ell})\qquad\textrm{and}\qquad G_{n,m,\ell} = \mathrm{Span}(v_{n,\ell+1},\ldots,v_{n,m})
$$
%Note that, until this step, we had not used yet the fact that $P_\ell$ is an orthogonal projection. But this assumption is instrumentally used in the following lemma in order to obtain some of the convergence properties.

\begin{lemma}\label{lemma4}
We have the following convergence properties:
\begin{align}
& \lim_{n\rightarrow+\infty}(\mathrm{id}-P_\ell)\Pi_{\overline D_{n,\ell}}=0,\qquad \lim_{n\rightarrow+\infty}A(\mathrm{id}-P_\ell)\Pi_{\overline D_{n,\ell}}=0 \label{operator2} \\
%& \lim_{n\rightarrow+\infty}(\mathrm{id}-\Pi_{\overline D_{n,\ell}})P_\ell=0 \label{operator22}Ê\\
& \lim_{n\rightarrow+\infty}\Pi_{\overline D_{n,\ell}}=P_\ell \label{operator3}\\
& \lim_{n\rightarrow+\infty}\Pi_{G_{n,m,\ell}}P_\ell=0 , \qquad \lim_{n\rightarrow+\infty}P_\ell\Pi_{G_{n,m,\ell}}=0  \label{operator33}
\end{align}
and
\begin{equation}\label{operator45}
\begin{split}
& \lim_{n\rightarrow+\infty}\Pi_{\overline D_{n,\ell}} A\Pi_{\overline D_{n,\ell}} =P_\ell AP_\ell, \qquad \lim_{n\rightarrow+\infty}\Pi_{G_{n,m,\ell}} A\Pi_{\overline D_{n,\ell}}=0 Ê\\
& \lim_{n\rightarrow+\infty}\Pi_{\overline D_{n,\ell}} A\Pi_{G_{n,m,\ell}}=0,\qquad \lim_{n\rightarrow+\infty}\Pi_{G_{n,m,\ell}} P_\ell AP_\ell\Pi_{G_{n,m,\ell}}=0
\end{split}
\end{equation}
\end{lemma}

\begin{remark}
According to \eqref{operator3}, we have $\overline D_{n,\ell}\simeq E_\ell$ when $n$ is large enough, as announced at the end of Section \ref{sec:strategy}.
\end{remark}

\begin{proof}
Recall that $\Pi_{\overline D_{n,\ell}}=V_{n,\ell}V_{n,\ell}^*$ where $V_{n,\ell}=(v_{n,1},\ldots,v_{n,\ell})$ is a matrix whose columns form an orthonormal basis of $\overline D_{n,\ell}$ (see Section \ref{sec:SVD}). Then it suffices to prove that $(\mathrm{id}-P_\ell)\Pi_{\overline D_{n,\ell}}V_{n,\ell}\rightarrow 0$ and $A(\mathrm{id}-P_\ell)\Pi_{\overline D_{n,\ell}}V_{n,\ell}\rightarrow 0$ as $n\rightarrow+\infty$. But $\Pi_{\overline D_{n,\ell}}V_{n,\ell}=V_{n,\ell}=Y_{n,\ell}U_{n,\ell}\Sigma_{n,\ell}^{-1}$, and \eqref{operator2} follows by \eqref{ineq_sec_proof_lem_optvalue} and \eqref{17}.

By Lemma \ref{lem_CVproj}, we have $\Pi_{\overline D_{n,\ell}}P_\ell-P_\ell\rightarrow 0$ (exponentially). Taking the adjoint (and using there that $P_\ell=P_\ell^*$), we obtain that $P_\ell\Pi_{\overline D_{n,\ell}}-P_\ell\rightarrow 0$. Since $\Pi_{\overline D_{n,\ell}}-P_\ell\Pi_{\overline D_{n,\ell}}\rightarrow 0$ by \eqref{operator2}, we infer that $\Pi_{\overline D_{n,\ell}}\rightarrow P_\ell$, i.e., \eqref{operator3} is established.

Still by Lemma \ref{lem_CVproj}, using that $\Pi_{\overline D_{n,m}}P_\ell-P_\ell\rightarrow 0$ and that $\Pi_{\overline D_{n,m}}=\Pi_{\overline D_{n,\ell}}+\Pi_{G_{n,m,\ell}}$ by definition, we obtain $\Pi_{G_{n,m,\ell}}P_\ell\rightarrow 0$, and then $P_\ell\Pi_{G_{n,m,\ell}}\rightarrow 0$ by taking the adjoint. We have proved \eqref{operator33}.

Let us now establish \eqref{operator45}. Using \eqref{decompA}, we have
$$
\Pi_{\overline D_{n,\ell}} A\Pi_{\overline D_{n,\ell}} = \Pi_{\overline D_{n,\ell}} P_\ell AP_\ell\Pi_{\overline D_{n,\ell}} + \Pi_{\overline D_{n,\ell}} (\mathrm{id}-P_\ell) A(\mathrm{id}-P_\ell)\Pi_{\overline D_{n,\ell}}
$$
Using \eqref{operator3}, the first term at the right-hand side converges to $P_\ell AP_\ell$. Using the second part of \eqref{operator2} and the fact that $\Vert\Pi_{\overline D_{n,\ell}} (\mathrm{id}-P_\ell)\Vert\leq 1$, we infer that the second term converges to $0$. Hence $\Pi_{\overline D_{n,\ell}} A\Pi_{\overline D_{n,\ell}} \rightarrow P_\ell AP_\ell$.

Using \eqref{decompA} again, we have
$$
\Pi_{G_{n,m,\ell}} A\Pi_{\overline D_{n,\ell}} = \Pi_{G_{n,m,\ell}} P_\ell P_\ell AP_\ell\Pi_{\overline D_{n,\ell}} + \Pi_{G_{n,m,\ell}} (\mathrm{id}-P_\ell) A(\mathrm{id}-P_\ell)\Pi_{\overline D_{n,\ell}}
$$
The first term at the right-hand side converges to $0$ because the range of $P_\ell AP_\ell\Pi_{\overline D_{n,\ell}}$ is finite-dimensional and $\Pi_{G_{n,m,\ell}} P_\ell\rightarrow 0$ by \eqref{operator33}. The second term converges to $0$ because $A(\mathrm{id}-P_\ell)\Pi_{\overline D_{n,\ell}}\rightarrow 0$ by \eqref{operator2} and $\Vert\Pi_{G_{n,m,\ell}} (\mathrm{id}-P_\ell)\Vert\leq 1$. Hence $\Pi_{G_{n,m,\ell}} A\Pi_{\overline D_{n,\ell}}\rightarrow 0$.

Taking the adjoint and using the fact that $A$ is selfadjoint (this is the only place where this assumption is actually crucial), we obtain that $\Pi_{\overline D_{n,\ell}} A\Pi_{G_{n,m,\ell}}\rightarrow 0$.

%Using \eqref{decompA} again, we have
%$$
%\Pi_{\overline D_{n,\ell}} A\Pi_{G_{n,m,\ell}} = \Pi_{\overline D_{n,\ell}} P_\ell AP_\ell P_\ell\Pi_{G_{n,m,\ell}} + \Pi_{\overline D_{n,\ell}} (\mathrm{id}-P_\ell) A(\mathrm{id}-P_\ell)\Pi_{G_{n,m,\ell}}
%$$
%The first term at the right-hand side converges to $0$ because $P_\ell\Pi_{G_{n,m,\ell}}\rightarrow 0$ by \eqref{operator33} and $\Vert \Pi_{\overline D_{n,\ell}} P_\ell AP_\ell\Vert\leq \Vert P_\ell AP_\ell\Vert$ is uniformly bounded. The second term converges to $0$ because $A(\mathrm{id}-P_\ell)\Pi_{G_{n,m,\ell}}$ is uniformly bounded (indeed, $G_{n,m,\ell}\subset\overline D_{n,m}\subset D_n=\mathrm{Ran}(Y_n)$ and then boundedness follows from \eqref{ineq_sec_proof_lem_optvalue}) and $\Pi_{\overline D_{n,\ell}} (\mathrm{id}-P_\ell)\rightarrow P_\ell (\mathrm{id}-P_\ell) = 0$. Hence $\Pi_{\overline D_{n,\ell}} A\Pi_{G_{n,m,\ell}}\rightarrow 0$.

Finally, $\Pi_{G_{n,m,\ell}} P_\ell AP_\ell\Pi_{G_{n,m,\ell}} = \Pi_{G_{n,m,\ell}} P_\ell AP_\ell P_\ell\Pi_{G_{n,m,\ell}}$ converges to $0$ because $P_\ell\Pi_{G_{n,m,\ell}}\rightarrow 0$ by \eqref{operator33} and $\Vert\Pi_{G_{n,m,\ell}} P_\ell AP_\ell \Vert\leq  \Vert P_\ell AP_\ell\Vert$ is uniformly bounded.
\end{proof}

\begin{lemma}\label{lemma5}
If $m\geq\ell$ then the POD reduced-order control system \eqref{reduced_control_system} is stabilizable whenever $n$ is large enough.
\end{lemma}

\begin{proof}
Without loss of generality, we assume that $m>\ell$. In the decomposition \eqref{decompDnm}, we have
$$A_{n,m} =
\begin{pmatrix}
\Pi_{\overline D_{n,\ell}} A\Pi_{\overline D_{n,\ell}} & \Pi_{\overline D_{n,\ell}} A\Pi_{G_{n,m,\ell}} \\
\Pi_{G_{n,m,\ell}} A\Pi_{\overline D_{n,\ell}} & \Pi_{G_{n,m,\ell}} A\Pi_{G_{n,m,\ell}}
\end{pmatrix}
\qquad\textrm{and}\qquad
B_{n,m} = \begin{pmatrix}
\Pi_{\overline D_{n,\ell}}B
\\ \Pi_{G_{n,m,\ell}}B
\end{pmatrix} .
$$
Using Lemma \ref{lemma4} and the fact that
\begin{multline*}
\Pi_{G_{n,m,\ell}} A\Pi_{G_{n,m,\ell}} = \Pi_{G_{n,m,\ell}} P_\ell AP_\ell\Pi_{G_{n,m,\ell}} + \Pi_{G_{n,m,\ell}} (\mathrm{id}-P_\ell)A(\mathrm{id}-P_\ell)\Pi_{G_{n,m,\ell}} \\
\sim \Pi_{G_{n,m,\ell}} (\mathrm{id}-P_\ell)A(\mathrm{id}-P_\ell)\Pi_{G_{n,m,\ell}}
\end{multline*}
we infer that
\begin{equation}\label{equivAB}
A_{n,m} \sim \begin{pmatrix}
P_\ell AP_\ell & 0 \\ 0 & \Pi_{G_{n,m,\ell}} (\mathrm{id}-P_\ell)A(\mathrm{id}-P_\ell)\Pi_{G_{n,m,\ell}}
\end{pmatrix}
\qquad\textrm{and}\qquad
B_{n,m} = \begin{pmatrix}
P_\ell B\\
\Pi_{G_{n,m,\ell}}B
\end{pmatrix}
\end{equation}
as $n\rightarrow+\infty$.
Assumption \ref{H2} implies (by a Laplace transform argument) that %$\Real(\mathrm{Spec}((\mathrm{id}-P_\ell)A(\mathrm{id}-P_\ell))\leq -\gamma$
any (real-valued) eigenvalue $\lambda$ of $(\mathrm{id}-P_\ell)A(\mathrm{id}-P_\ell)$ satisfies $\lambda\leq-\gamma$, hence the matrix $\Pi_{G_{n,m,\ell}} (\mathrm{id}-P_\ell)A(\mathrm{id}-P_\ell)\Pi_{G_{n,m,\ell}}$ is uniformly Hurwitz with respect to $n$ large enough, by Lemma \ref{lem_H} hereafter.
Since the pair $(P_\ell AP_\ell,P_\ell B)$ satisfies the Kalman condition (Assumption \ref{H4}), the conclusion is now obvious.
\end{proof}

\begin{lemma}\label{lem_H}
The matrix $\Pi_{G_{n,m,\ell}} (\mathrm{id}-P_\ell)A(\mathrm{id}-P_\ell)\Pi_{G_{n,m,\ell}}$ is uniformly Hurwitz with respect to $n$.
\end{lemma}

\begin{proof}
It follows from Assumption \ref{H2} that the operator $(\mathrm{id}-P_\ell)(A+\gamma\,\mathrm{id})(\mathrm{id}-P_\ell)$ generates the contraction $C_0$ semigroup $\left( (\mathrm{id}-P_\ell) e^{\gamma t} S(t) (\mathrm{id}-P_\ell) \right)_{t\geq 0}$. By the converse of the Lumer-Phillips theorem (see \cite[Chapter 1, Section 1.4, Theorem 4.3 and Corollary 4.4]{Pazy}), the (selfadjoint) operator $(\mathrm{id}-P_\ell)(A+\gamma\,\mathrm{id})(\mathrm{id}-P_\ell)$ is dissipative, which means that
$$
%\Real
\langle (\mathrm{id}-P_\ell)(A+\gamma\,\mathrm{id})(\mathrm{id}-P_\ell)y,y\rangle \leq 0\qquad \forall y\in H .
$$
%and we have the same property for its adjoint.
Hence
$$
%\Real
\langle \Pi_{G_{n,m,\ell}}(\mathrm{id}-P_\ell)(A+\gamma\,\mathrm{id})(\mathrm{id}-P_\ell)\Pi_{G_{n,m,\ell}}y,y\rangle \leq 0\qquad \forall y\in H .
$$
%and the same inequality for the adjoint.
By the Lumer-Phillips theorem, $\Pi_{G_{n,m,\ell}}(\mathrm{id}-P_\ell)(A+\gamma\,\mathrm{id})(\mathrm{id}-P_\ell)\Pi_{G_{n,m,\ell}}$ generates a contraction semigroup, hence, using \eqref{operator33} and the fact that $\Pi_{G_{n,m,\ell}}(\mathrm{id}-P_\ell)(\mathrm{id}-P_\ell)\Pi_{G_{n,m,\ell}}\sim \Pi_{G_{n,m,\ell}}$ as $n\rightarrow+\infty$, we infer that the semigroup generated by $\Pi_{G_{n,m,\ell}} (\mathrm{id}-P_\ell)A(\mathrm{id}-P_\ell)\Pi_{G_{n,m,\ell}}$ decreases exponentially (like $e^{-\gamma t}$). The lemma is proved.
\end{proof}

\subsection{Proof of Theorem \ref{mainthm}}\label{sec:proofs}
We are now in a position to prove Theorem \ref{mainthm}.
With respect to the decomposition $H =E_\ell\overset{\bot}{\oplus} F_\ell$, thanks to \eqref{spectralproperties}, the operator $A + B K_{n,m}(\varepsilon) \Pi_{\overline D_{n,m}}$ is written as the infinite-dimensional matrix
\begin{equation}\label{splits}
%A + B K_{n,m}(\varepsilon)  \Pi_{\overline D_{n,m}} =
\left(\begin{array}{c|c}
P_\ell (A + B K_{n,m}(\varepsilon)  \Pi_{\overline D_{n,m}}) P_\ell & P_\ell B K_{n,m}(\varepsilon)  \Pi_{\overline D_{n,m}} (\mathrm{id}-P_\ell) \\[2mm] \hline
\\[-2mm]
(\mathrm{id}-P_\ell)  B K_{n,m}(\varepsilon)  \Pi_{\overline D_{n,m}}P_\ell & (\mathrm{id}-P_\ell) (A + B K_{n,m}(\varepsilon)  \Pi_{\overline D_{n,m}}) (\mathrm{id}-P_\ell)
\end{array}\right)
\end{equation}
and the exponential stability of \eqref{stateclosedloop} is equivalent to that of \eqref{splits}.

The idea is the following.
By Proposition \ref{prop_riccati2} in Appendix \ref{app:riccati}, when $\varepsilon$ is small and $n$ is large enough, the feedback matrix $K_{n,m}(\varepsilon)$ essentially acts on the $\ell$ first modes that are instable, and  $K_{n,m}(\varepsilon) \Pi_{\overline D_{n,m}} (\mathrm{id}-P_\ell)$ converges to $0$ as $\varepsilon\rightarrow 0$ and $n\rightarrow+\infty$. Therefore, if $\varepsilon$ is small enough and $n$ is large enough then \eqref{splits} is almost lower block triangular, with the $\ell\times\ell$ block at the top left being uniformly Hurwitz and the infinite-dimensional block at the bottom right (close to $(\mathrm{id}-P_\ell) A (\mathrm{id}-P_\ell)$) generating a uniformly exponentially stable semigroup.

We now give the detail of the arguments.

\begin{lemma}\label{lemm7}
Given any $m\geq \ell$, the $\ell\times\ell$ matrix $P_\ell (A + B K_{n,m}(\varepsilon)  \Pi_{\overline D_{n,m}}) P_\ell$ is uniformly Hurwitz %as $\varepsilon\rightarrow 0$ and $n\rightarrow+\infty$.
and
\begin{equation}\label{K_nm}
K_{n,m}(\varepsilon)  \Pi_{\overline D_{n,m}} (\mathrm{id}-P_\ell) = \mathrm{o}(1)
\end{equation}
as $\varepsilon\rightarrow 0$ and $n\rightarrow +\infty$.
\end{lemma}

\begin{proof}
Without loss of generality, we assume that $m>\ell$.
Let us consider the POD reduced-order control system \eqref{reduced_control_system}.
In the decomposition $\overline D_{n,m}=\overline D_{n,\ell}\overset{\bot}{\oplus}G_{n,m,\ell}$ (see \eqref{decompDnm}), we have
$$
A_{n,m} =\begin{pmatrix}
\Pi_{\overline D_{n,\ell}} A\Pi_{\overline D_{n,\ell}}  & \Pi_{\overline D_{n,\ell}} A\Pi_{G_{n,m,\ell}}\\
\Pi_{G_{n,m,\ell}} A\Pi_{\overline D_{n,\ell}} &\Pi_{G_{n,m,\ell}} A\Pi_{G_{n,m,\ell}}
\end{pmatrix}
\qquad\textrm{and}\qquad
B_{n,m}=\begin{pmatrix}
\Pi_{\overline D_{n,\ell}}B\\
\Pi_{G_{n,m,\ell}} B
\end{pmatrix}
$$
and accordingly, the maximal positive semidefinite solution $P_{n,m}(\varepsilon)$ of the algebraic Riccati equation \eqref{RiccatiPnm} is written as
$$
P_{n,m}(\varepsilon) = \begin{pmatrix}
P_{n,m}^1(\varepsilon) & P_{n,m}^3(\varepsilon) \\ P_{n,m}^3(\varepsilon)^* & P_{n,m}^2(\varepsilon)
\end{pmatrix}
$$
and
\begin{equation*}
\begin{split}
K_{n,m}(\varepsilon)
&= -B_{n,m}^*P_{n,m}(\varepsilon) \\
&= - \left( B^*\Pi_{\overline D_{n,\ell}} P_{n,m}^1(\varepsilon) + B^*\Pi_{G_{n,m,\ell}} P_{n,m}^3(\varepsilon)^* ,  B^*\Pi_{\overline D_{n,\ell}} P_{n,m}^3(\varepsilon) + B^*\Pi_{G_{n,m,\ell}} P_{n,m}^2(\varepsilon) \right)
\end{split}
\end{equation*}
We apply Proposition \ref{prop_riccati2} in Appendix \ref{app:riccati}, roughly speaking, with $\alpha=\frac{1}{n}$. More precisely, the matrices $A(\alpha)$ and $B(\alpha)$ of Appendix \ref{app:riccati} are any continuous interpolations at $\alpha=\frac{1}{n}$, $n\in\N^*$, of the matrices $A_{n,m}$ and $B_{n,m}$. By Lemma \ref{lemma4}, all assumptions done in Appendix \ref{app:riccati} are satisfied, because we have \eqref{equivAB} with $-P_\ell AP_\ell$ being Hurwitz and $\Pi_{G_{n,m,\ell}} (\mathrm{id}-P_\ell)A(\mathrm{id}-P_\ell)\Pi_{G_{n,m,\ell}}$ being uniformly Hurwitz (by Lemma \ref{lem_H}) and the pair $(P_\ell AP_\ell, P_\ell B)$ satisfies the Kalman condition by \ref{H4}.
According to Proposition \ref{prop_riccati2}, we have $P_{n,m}(\varepsilon)\rightarrow P_{\infty,m}(0)$ as $\varepsilon\rightarrow 0$ and $n\rightarrow +\infty$, with $P_{\infty,m}^1(0)$ symmetric positive semidefinite, and $A_{n,m}-B_{n,m}B_{n,m}^*P_{n,m}(\varepsilon)$ is uniformly Hurwitz with respect to $\varepsilon$ small enough and $n$ large enough. Moreover, by \eqref{a.55}, using \eqref{operator3}, we have
$$
K_{n,m}(\varepsilon)  = - \left( B^*P_\ell P_{\infty,m}^1(0) ,  0  \right) + \mathrm{o}(1)
$$
as $\varepsilon\rightarrow 0$ and $n\rightarrow +\infty$.
This is more than required to obtain \eqref{K_nm}.
Moreover, using \eqref{equivAB}, we have
\begin{equation}\label{uselemm7}
A_{n,m}+B_{n,m}K_{n,m}(\varepsilon) \sim
\begin{pmatrix}
P_\ell ( A - BB^*P_\ell P_{\infty,m}^1(0) ) P_\ell & 0 \\
\Pi_{G_{n,m,\ell}}B B^*P_\ell P_{\infty,m}^1(0) & \Pi_{G_{n,m,\ell}} (\mathrm{id}-P_\ell)A(\mathrm{id}-P_\ell)\Pi_{G_{n,m,\ell}}
\end{pmatrix}
\end{equation}
as $\varepsilon\rightarrow 0$ and $n\rightarrow +\infty$.
But the block at the top left is equivalent to $P_\ell (A + B K_{n,m}(\varepsilon)  \Pi_{\overline D_{n,m}}) P_\ell$, indeed:
$$
P_\ell (A + B K_{n,m}(\varepsilon)  \Pi_{\overline D_{n,m}}) P_\ell
= P_\ell ( A - B B^* P_\ell P_{\infty,m}^1(0) ) P_\ell +\mathrm{o}(1)
$$
as $\varepsilon\rightarrow 0$ and $n\rightarrow +\infty$.

It follows that $P_\ell (A + B K_{n,m}(\varepsilon)  \Pi_{\overline D_{n,m}}) P_\ell$, which is the block on the top left of \eqref{splits}, is uniformly Hurwitz for $\varepsilon$ small enough and $n$ large enough.
\end{proof}

Theorem \ref{mainthm} is proved.

To finish, let us prove \eqref{decayrate} in Remark \ref{rem_exprate}. As said above, in approximation, the infinite-dimensional matrix \eqref{splits} is lower block triangular. The block at the bottom right is approximately equal to $(\mathrm{id}-P_\ell)A(\mathrm{id}-P_\ell)$, which generates a semigroup that is exponentially stable, with decay rate $\gamma$ (given by \ref{H2}).
The block at the top left is treated thanks to the proof of Lemma \ref{lemm7}. By \eqref{uselemm7}, $P_\ell(A+BK_{n,m}(\varepsilon)\Pi_{\overline D_{n,m}})P_\ell$ is approximately equal to the block at the top left of the lower block triangular matrix $A_{n,m}+B_{n,m}K_{n,m}(\varepsilon)$. Therefore, the spectral abscissa of $A_{n,m}+B_{n,m}K_{n,m}(\varepsilon)$ is approximately equal to the minimum of the spectral abscissa of $P_\ell(A+BK_{n,m}(\varepsilon)\Pi_{\overline D_{n,m}})P_\ell$ and of the spectral abscissa of $\Pi_{G_{n,m,\ell}}(\mathrm{id}-P_\ell)A(\mathrm{id}-P_\ell)\Pi_{G_{n,m,\ell}}$ (which is greater than or equal to $\gamma$).
The formula \eqref{decayrate} follows.

\section{Conclusion and perspectives}\label{sec:ccl}
Considering a linear autonomous control system in infinite dimension, the Proper Orthogonal Decomposition (POD) approach generates a finite-dimensional reduced-order model, which can be used to generate low-order controls. Applying the Riccati theory to the reduced-order control system leads to a finite-dimensional linear stabilizing feedback. In this paper, we have proved that, under appropriate assumptions, this reduced-order feedback exponentially stabilizes as well the whole infinite-dimensional control system.

Our assumptions are of spectral nature. The operator $A$ underlying the control system is assumed to be selfadjoint, and the system with null control is assumed to have a finite-dimensional instable part and an infinite-dimensional stable part. Our assumptions involve the case of heat-like equations with internal control, but not the case of damped wave equations.

Several comments and open issues are in order.

\begin{itemize}

\item In the POD method, we have taken snapshots at times $t_k=kT$. The choice of such regular snapshots is compatible with the application of the Kalman condition in our proofs. Treating the case of other, more random snapshots $t_k$ is an interesting issue.

\item It is likely that, under controllability assumptions, one can obtain a result of rapid stabilization in the following sense: given any $\gamma>0$, find a low-order feedback matrix $K_{n,m}^\gamma$, designed from a POD approximation, such that $\|y(t)\|\leq C(\gamma)e^{-\gamma t}\|y(0)\|$ for every $t\geq 0$.
We let this issue for further investigations.

\item The spectral assumptions \ref{H1}, \ref{H2}, \ref{H3} done at the beginning of the paper typically apply to a parabolic partial differential equation with a selfadjoint operator, like heat, anomalous equations or Stokes equations. As mentioned above, our assumptions do not involve the case of damped wave equations, for which
$$
A=\begin{pmatrix} 0 & \mathrm{id} \\ \triangle + a\,\mathrm{id} & - b\, \mathrm{id} \end{pmatrix}
$$
with $b> 0$ on $\Omega$, is neither selfadjoint nor normal.
Treating such operators is a major open issue.

\item Throughout the paper we have considered bounded control operators $B$. We are thus able to treat the case of internal controls, but not, in general, the case of boundary controls. The case where $B\in L(U,D(A^*)')$ is unbounded but admissible (see \cite{TucsnakWeiss}) is open.

\item We have considered linear control systems, but, since POD can be applied as well to nonlinear systems, we expect that similar results may hold true for semilinear control systems, of the form
$$
\dot y(t) = Ay(t) + F(y(t)) + Bu(t)
$$
where $F:H\rightarrow H$ is a nonlinear mapping of class $C^1$, Lipschitz on the bounded sets of $H$, with $F(0)=0$ (see \cite{AlabauPrivatTrelat} for stabilization results and discretization issues). Adapting our approach to that framework is another interesting open issue.
\end{itemize}

%\item POD and shape optimization

\appendix

\section{Appendix: an asymptotic result in Riccati theory}\label{app:riccati}
Let $\ell$, $m$ and $p$ be integers satisfying $0<\ell< m$. Hereafter, we always denote by $\Vert\cdot\Vert$ the Euclidean norm, may it be in $\R^\ell$, in $\R^{m-\ell}$, in $\R^m$ or in $\R^p$.
We consider the family of autonomous linear control systems in $\R^{m}$, indexed by $\alpha\in(0,\alpha_0]$ for some $\alpha_0>0$,
\begin{equation}\label{app:sysalpha}
\begin{split}
\dot x_1 &= A_1(\alpha) x_1 + A_3(\alpha) x_2 + B_1(\alpha) u  \\
\dot x_2 &= A_4(\alpha) x_1 + A_2(\alpha) x_2 + B_2(\alpha) u
\end{split}
\end{equation}
with matrices $A_1(\alpha)$ of size $\ell \times  \ell$, $A_2(\alpha)$ of size $(m-\ell)\times (m-\ell)$, $A_3(\alpha)$ of size $\ell \times  (m-\ell)$, $A_4(\alpha)$ of size $(m-\ell) \times  \ell$, $B_1(\alpha)$ of size $\ell\times p$ and $B_2(\alpha)$ of size $(m-\ell)\times p$, depending continuously on $\alpha$ and satisfying
\begin{equation}\label{assum}
\begin{split}
& A_1(\alpha)=A_1+\underset{\alpha\rightarrow 0}{\mathrm{o}}(1),
\qquad A_3(\alpha)=\underset{\alpha\rightarrow 0}{\mathrm{o}}(1),\qquad A_4(\alpha)=\underset{\alpha\rightarrow 0}{\mathrm{o}}(1), \\
& B_1(\alpha)=B_1+\underset{\alpha\rightarrow 0}{\mathrm{o}}(1), \\
& \Vert B_2(\alpha)\Vert\ \textrm{uniformly bounded with respect to}\  \alpha\in(0,\alpha_0],
\end{split}
\end{equation}
with $A_1$ of size $\ell\times\ell$ and $B_1$ of size $\ell\times p$.
We assume that:
\begin{itemize}
\item $-A_1(\alpha)$ and $A_2(\alpha)$ are uniformly Hurwitz with respect to $\alpha\in(0,\alpha_0]$, i.e., there exists $\eta>0$ such that, for every $\alpha\in(0,\alpha_0]$, all eigenvalues $\lambda$ of $-A_1(\alpha)$ and of $A_2(\alpha)$ satisfy $\Real(\lambda)\leq -\eta$.

This means that, for $u=0$, the first part of the system \eqref{app:sysalpha} is instable, while the second part is asymptotically stable, uniformly with respect to $\alpha\in(0,\alpha_0]$.

\item The pair $(A_1,B_1)$ satisfies the Kalman condition.
\end{itemize}
Note that this implies that the pair $(A_1(\alpha),B_1(\alpha))$ satisfies the Kalman condition for every $\alpha\in(0,\alpha_0]$, provided $\alpha_0$ is small enough.
We set
$$
A(\alpha) = \begin{pmatrix}
A_1(\alpha) & A_3(\alpha) \\ A_4(\alpha) & A_2(\alpha)
\end{pmatrix}
\qquad\textrm{and}\qquad
B(\alpha) = \begin{pmatrix}
B_1(\alpha) \\ B_2(\alpha)
\end{pmatrix} .
$$
Of course, it follows from \eqref{assum} that, as soon as $\alpha_0$ is small enough, given any $\alpha\in(0,\alpha_0]$, $A(\alpha)$ has no pure imaginary eigenvalue and the pair $(A(\alpha),B(\alpha))$ is stabilizable.
Then, in order to stabilize \eqref{app:sysalpha}, it suffices to stabilize the first part of the system, by choosing $u=K_1x_1$ such that $A_1+B_1K_1$ is Hurwitz (this is possible because the pair $(A_1,B_1)$ satisfies the Kalman condition), and then clearly this control stabilizes as well the whole system \eqref{app:sysalpha} for any $\alpha$ small enough.

But here, following the spirit of the method developed in this paper, we ``forget" that we know how to split the above system in a stable part and an instable part, and hereafter we propose to stabilize the control system \eqref{app:sysalpha} ``blindly", i.e., ignoring the exact splitting, by using the classical Riccati theory, as follows.

\paragraph{Algebraic Riccati equation.}
Let $\varepsilon\geq 0$ and $\alpha\in(0,\alpha_0]$ be arbitrary. Since the pair $(A(\alpha),B(\alpha))$ is stabilizable, by \cite[Corollary 2.3.7 page 55]{Abou-Kandil} there exists a (unique) maximal symmetric positive semidefinite matrix (of size $m\times m$)
$$
P(\varepsilon,\alpha) = \begin{pmatrix}
P_1(\varepsilon,\alpha) & P_3(\varepsilon,\alpha) \\ P_3(\varepsilon,\alpha)^* & P_2(\varepsilon,\alpha)
\end{pmatrix}
$$
solution of the algebraic Riccati equation
\begin{equation}\label{ricca}
A(\alpha)^* P(\varepsilon,\alpha) + P(\varepsilon,\alpha)A(\alpha) - P(\varepsilon,\alpha) B(\alpha) B(\alpha)^* P(\varepsilon,\alpha) + \varepsilon I_{m} = 0  .
\end{equation}
Moreover:
\begin{itemize}
\item $A(\alpha)-B(\alpha)B(\alpha)^*P(\varepsilon,\alpha)$ is semi-stabilizing, i.e., its eigenvalues have nonpositive real part;
\item if $\varepsilon>0$ then $P(\varepsilon,\alpha)$ is positive definite, and $A(\alpha)-B(\alpha)B(\alpha)^*P(\varepsilon,\alpha)$ is Hurwitz, i.e., its eigenvalues have negative real part.\footnote{This is a consequence of \cite[Chapter 2, Corollary 2.4.3 page 60 and Theorem 2.4.25 page 77]{Abou-Kandil}. Indeed, since $A(\alpha)$ has no pure imaginary eigenvalue, the Hamiltonian matrix
$$
\begin{pmatrix} A(\alpha) & -B(\alpha)^* B(\alpha) \\ -\varepsilon I_m & -A(\alpha)^* \end{pmatrix}
$$
has no pure imaginary eigenvalue.}
\end{itemize}
As a consequence, setting
$$
K(\varepsilon,\alpha)= - B(\alpha)^* P(\varepsilon,\alpha) ,\qquad x=\begin{pmatrix} x_1\\Êx_2\end{pmatrix},
$$
the linear feedback
\begin{multline}\label{def_feedback}
\hat u_{\varepsilon,\alpha} = K(\varepsilon,\alpha)x= - B(\alpha)^* P(\varepsilon,\alpha) x \\
= - \left( B_1^* P_1(\varepsilon,\alpha) + B_2^* P_3(\varepsilon,\alpha)^* \right) x_1 - \left( B_1^* P_3(\varepsilon,\alpha) + B_2^* P_2(\varepsilon,\alpha) \right) x_2
\end{multline}
exponentially stabilizes the system \eqref{app:sysalpha} to the origin.

\paragraph{Relationship with linear quadratic optimal control.}
Throughout, for every $\alpha\in(0,\alpha_0]$ we denote by
$$
x(t;\alpha,u,x_1(0),x_2(0)) =
\begin{pmatrix}
x_1(t;\alpha,u,x_1(0),x_2(0))\\
x_2(t;\alpha,u,x_1(0),x_2(0))
\end{pmatrix}
$$
the solution of \eqref{app:sysalpha} with initial condition $(x_1(0),x_2(0))$ and with control $u$.

According to the well known linear quadratic Riccati theory (see, e.g., \cite{Abou-Kandil, kwak, LeeMarkus, Sontag, Trelat}), given any $\varepsilon\geq 0$ and any $\alpha\in(0,\alpha_0]$, given any initial condition $(x_{1,0},x_{2,0})\in\R^\ell\times\R^{m-\ell}$, there exists a unique optimal control minimizing the cost functional
\begin{equation}\label{121960}
J_{\varepsilon,\alpha}(u; x_{1,0},x_{2,0})= \int_0^{+\infty} \left( \varepsilon
\Vert x_1(t;\alpha,u,x_{1,0},x_{2,0})\Vert^2 +  \varepsilon \Vert x_2(t;\alpha,u,x_{1,0},x_{2,0})\Vert^2 +\Vert u(t)\Vert^2 \right) dt
\end{equation}
over all possible controls $u\in L^2(0,+\infty; \R^p)$.
If $\varepsilon=0$ then the optimal control is $u=0$. If $\varepsilon>0$ then the optimal control coincides with the feedback control $\hat u_{\varepsilon,\alpha}$ defined by \eqref{def_feedback}, which exponentially stabilizes the control system \eqref{app:sysalpha}. Moreover, $\hat u_{\varepsilon,\alpha}$ is linear with respect to the initial condition $(x_{1,0},x_{2,0})$, and we have
\begin{multline*}
J_{\varepsilon,\alpha}(\hat u_{\varepsilon,\alpha}; x_{1,0},x_{2,0})= \begin{pmatrix} x_{1,0} \\ x_{2,0}\end{pmatrix}^* P(\varepsilon,\alpha) \begin{pmatrix} x_{1,0} \\ x_{2,0}\end{pmatrix} \\
= x_{1,0}^* P_1(\varepsilon,\alpha) x_{1,0} + 2 x_{1,0}^* P_3(\varepsilon,\alpha) x_{2,0} + x_{2,0}^* P_2(\varepsilon,\alpha) x_{2,0} .
\end{multline*}

\medskip

Our objective is to investigate the asymptotics of $P(\varepsilon,\alpha)$ as $\varepsilon\rightarrow 0$. We are going to prove that $A(\alpha) - B(\alpha)B(\alpha)^* P(\varepsilon,\alpha)$ remains uniformly Hurwitz with respect to $(\varepsilon,\alpha)$ small enough and that the optimal feedback $K(\varepsilon,\alpha) = -B(\alpha)^*P(\varepsilon,\alpha)$ essentially acts on the $\ell$ first modes that are instable (eigenvalues of $A_1(\alpha)$), in the sense that $K(\alpha) = -(B_1(\alpha)^*P_1(0,0),0) + \mathrm{o}(1)$ as $(\varepsilon,\alpha)\rightarrow 0$.

\subsection{A first result in the block diagonal case}\label{app:riccati1}
In this first subsection, we assume that
\begin{equation}\label{A3A4zero}
A_3(\alpha)=0 \qquad\textrm{and}\qquad A_4(\alpha)=0
\end{equation}
for every $\alpha\in(0,\alpha_0]$, i.e., that $A(\alpha)$ is block diagonal.

\begin{proposition}\label{prop_riccati1}
We have the following results:
\begin{itemize}
\item For every $\alpha\in(0,\alpha_0]$, $P(\varepsilon,\alpha)$ is continuous with respect to $\varepsilon\geq 0$. Moreover, we have $P_2(0,\alpha)=0$ and $P_3(0,\alpha)=0$ for every $\alpha\in(0,\alpha_0]$ and
\begin{equation}\label{a.54}
P_1(\varepsilon,\alpha) = P_1(0,\alpha) + \underset{\varepsilon\rightarrow 0}{\mathrm{o}}(1) ,\qquad
P_2(\varepsilon,\alpha) = \underset{\varepsilon\rightarrow 0}{\mathrm{o}}(1),\qquad
P_3(\varepsilon,\alpha) = \underset{\varepsilon\rightarrow 0}{\mathrm{o}}(1)
\end{equation}
where the remainder terms are uniform with respect to $\alpha\in(0,\alpha_0]$. In particular, the mapping $(\varepsilon,\alpha)\mapsto P(\varepsilon,\alpha)$ has a continuous extension at $(0,0)$.
\item For every $\alpha\in[0,\alpha_0]$, the matrix $P_1(0,\alpha)$ is symmetric positive semidefinite (but not necessarily definite). % and the matrix $A_1(\alpha)-B_1(\alpha)B_1(\alpha)^* P_1(0,\alpha)$ is Hurwitz uniformly with respect to $\alpha\in(0,\alpha_0]$.
\item The matrix $A(\alpha) - B(\alpha)B(\alpha)^* P(\varepsilon,\alpha)$ is uniformly Hurwitz with respect to $(\varepsilon,\alpha)$ small enough, meaning that there exist $\varepsilon_0>0$ and $\eta>0$ such that, for every $(\varepsilon,\alpha)\in[0,\varepsilon_0]\times[0,\alpha_0]$, every eigenvalue $\lambda$ of $A(\alpha) - B(\alpha)B(\alpha)^* P(\varepsilon,\alpha)$ is such that $\Real(\lambda)\leq-\eta$.
%Moreover, $-\eta$ can be chosen as any negative real number which is greater than the maximum of the spectral abscissa of $A_1-B_1B_1^* P_1(0)$ and of $A_2$.
\end{itemize}
\end{proposition}

%Note that, as a consequence of \eqref{a.54}, the feedback matrix $K(\varepsilon,\alpha)=-B(\alpha)^* P(\varepsilon,\alpha)$ satisfies
%\begin{equation*}%\label{a.55}
%\begin{split}
%K(\varepsilon,\alpha) &= -(B_1(\alpha)^*P_1(\varepsilon,\alpha)+B_2(\alpha)^*P_3(\varepsilon,\alpha), B_1(\alpha)^*P_3(\varepsilon,\alpha)+B_2(\alpha)^*P_2(\varepsilon,\alpha)) \\
%&= -(B_1(\alpha)^*P_1(0,\alpha),0) + \underset{\varepsilon\rightarrow 0}{\mathrm{o}}(1)
%\end{split}
%\end{equation*}
%uniformly with respect to $\alpha\in(0,\alpha_0]$. In other words, for $\varepsilon\geq 0$ small the feedback matrix $K(\varepsilon,\alpha)$ essentially acts on the $\ell$ first modes that are instable (eigenvalues of $A_1(\alpha)$).

\begin{proof}
The proof goes in three steps.

\emph{Step 1.}
The fact that for $\alpha$ fixed $P(\varepsilon,\alpha)$ depends continuously on $\varepsilon$ is completely general. We sketch the proof. Let $\alpha\in(0,\alpha_0]$ be fixed. By the minimization property, the symmetric positive semidefinite matrix $P(\varepsilon,\alpha)$ is monotone increasing with respect to $\varepsilon$, hence $P(\varepsilon,\alpha)$ is uniformly bounded with respect to $\varepsilon\in[0,1]$.
Let $\varepsilon\geq 0$ be fixed and let $(\varepsilon_k)_{k\in\N^*}$ be a sequence of nonnegative real numbers converging to $\varepsilon$. Since $P(\varepsilon_k,\alpha)$ is bounded, up to some subsequence we have $P(\varepsilon_k,\alpha)\rightarrow \tilde P(\alpha)$ as $k\rightarrow+\infty$. Since $P(\varepsilon_k,\alpha)$ is the (unique) maximal symmetric positive semidefinite matrix solution of \eqref{ricca} (with $\varepsilon=\varepsilon_k$), we easily infer that $\tilde P(\alpha)$ is the maximal symmetric positive semidefinite matrix solution of \eqref{ricca}. Hence (by uniqueness) $\tilde P(\alpha)=P(\varepsilon,\alpha)$, and since the argument is valid for any sequence, the continuity property follows. In particular, we have $P(\varepsilon,\alpha)=P(0,\alpha)+\underset{\varepsilon\rightarrow 0}{\mathrm{o}}(1)$ as $\varepsilon\rightarrow 0$. Anyway this has been done for $\alpha$ fixed and at this step we do not know yet that $(\varepsilon,\alpha)\mapsto P(\varepsilon,\alpha)$ has a continuous extension at $(0,0)$.

Note that $P_1(\varepsilon,\alpha)$ and $P_2(\varepsilon,\alpha)$ are symmetric positive semidefinite for every $\varepsilon\geq 0$ and every $\alpha\in(0,\alpha_0]$, and positive definite when $\varepsilon>0$ (because $P(\varepsilon,\alpha)$ is so).

\medskip

\emph{Step 2.}
Let us prove that $P_2(\varepsilon,\alpha) = \underset{\varepsilon\rightarrow 0}{\mathrm{o}}(1)$ and that $P_3(\varepsilon,\alpha) = \underset{\varepsilon\rightarrow 0}{\mathrm{o}}(1)$ uniformly with respect to $\alpha\in(0,\alpha_0]$, provided $\alpha_0$ is small enough.

Given any $\varepsilon>0$, any $\alpha\in(0,\alpha_0]$ and any $x_{2,0}\in \R^{m-\ell}$, we have $x_{2,0}^* P_2(\varepsilon,\alpha) x_{2,0}\leq J_{\varepsilon,\alpha}(0; 0, x_{2,0})$ by the minimization property. Besides, using \eqref{A3A4zero} and the fact that $A_2(\alpha)$ is uniformly Hurwitz, we have
$$
J_{\varepsilon,\alpha}(0; 0, x_{2,0}) = \varepsilon \int_0^{+\infty} \Vert e^{tA_2(\alpha)}x_{2,0}\Vert^2\, dt\leq C\varepsilon\Vert x_{2,0}\Vert^2
$$
for some $C>0$ independent of $(\varepsilon,\alpha)$ and of $x_{2,0}$. It follows that $P_2(\varepsilon,\alpha) = \underset{\varepsilon\rightarrow 0}{\mathrm{o}}(1)$ uniformly with respect to $\alpha\in(0,\alpha_0]$.

Now, let $x_{1,0}\in\R^{\ell}$ and $x_{2,0}\in \R^{m-\ell}$ be arbitrary. For every $\varepsilon> 0$ and every $\alpha\in(0,\alpha_0]$, let $u_{\varepsilon,\alpha}=u_{\varepsilon,\alpha}(x_{1,0},0)$ be the solution of the minimization problem $\inf_u J_{\varepsilon,\alpha}(u; x_{1,0},0)$.
Then $J_{\varepsilon,\alpha}(u_{\varepsilon,\alpha}; x_{1,0},0)=x_{1,0}^* P_1(\varepsilon,\alpha) x_{1,0}$ and we have the following lemma.

\begin{lemma}\label{lemtech5}
We have $J_{\varepsilon,\alpha}(u_{\varepsilon,\alpha}; x_{1,0},x_{2,0})-J_{\varepsilon,\alpha}(u_{\varepsilon,\alpha}, x_{1,0},0) = \underset{\varepsilon\rightarrow 0}{\mathrm{o}}(1)\left(\|x_{1,0}\|^2+\|x_{2,0}\|^2\right)$ uniformly with respect to $\alpha\in(0,\alpha_0]$.
\end{lemma}

\begin{proof}
We have
\begin{equation*}
\begin{split}
J_{\varepsilon,\alpha}(u_{\varepsilon,\alpha},x_{1,0},x_{2,0}) &= \int_0^{+\infty}\big( \varepsilon\Vert x_1(t;\alpha,u_{\varepsilon,\alpha},x_{1,0},x_{2,0})\Vert^2 + \varepsilon\Vert x_2(t;\alpha,u_{\varepsilon,\alpha},x_{1,0},x_{2,0})\Vert^2 \\
&\qquad\qquad\qquad\qquad\qquad\qquad\qquad\qquad\qquad\qquad\qquad\qquad\qquad + \Vert u_{\varepsilon,\alpha}(t)\Vert^2 \big) dt \\
J_{\varepsilon,\alpha}(u_{\varepsilon,\alpha},x_{1,0},0) &= \int_0^{+\infty}\left( \varepsilon\Vert x_1(t;\alpha,u_{\varepsilon,\alpha},x_{1,0},0)\Vert^2 + \varepsilon\Vert x_2(t;\alpha,u_{\varepsilon,\alpha},x_{1,0},0)\Vert^2 + \Vert u_{\varepsilon,\alpha}(t)\Vert^2 \right) dt
\end{split}
\end{equation*}
with the first integral being possibly equal to $+\infty$ (this is not a problem and anyway we show below that it is finite).
Now, on the one part, using \eqref{A3A4zero}, we have $x_1(t;\alpha,u_{\varepsilon,\alpha},x_{1,0},x_{2,0}) = x_1(t;\alpha,u_{\varepsilon,\alpha},x_{1,0},0)$. On the other part, we infer from the Duhamel formula that
$$
x_2(t;\alpha,u_{\varepsilon,\alpha},x_{1,0},x_{2,0}) = e^{tA_2(\alpha)}x_{2,0} + x_2(t;\alpha,u_{\varepsilon,\alpha},x_{1,0},0) .
$$
Hence
\begin{equation}\label{16:09}
\begin{split}
J_{\varepsilon,\alpha}& (u_{\varepsilon,\alpha}; x_{1,0},x_{2,0})-J_{\varepsilon,\alpha}(u_{\varepsilon,\alpha}, x_{1,0},0)
\\
&= \varepsilon \int_0^{+\infty} \left( \Vert e^{tA_2(\alpha)}x_{2,0} + x_2(t;\alpha,u_{\varepsilon,\alpha},x_{1,0},0)\Vert^2 - \Vert x_2(t;\alpha,u_{\varepsilon,\alpha},x_{1,0},0)\Vert^2\right) dt \\
&=  \varepsilon \int_0^{+\infty} \left( \Vert e^{tA_2(\alpha)}x_{2,0}\Vert^2 + 2\langle e^{tA_2(\alpha)}x_{2,0},x_2(t;\alpha,u_{\varepsilon,\alpha},x_{1,0},0)\rangle \right) dt
\end{split}
\end{equation}
As already said, as a general fact of the Riccati theory, all eigenvalues of $A(\alpha)-B(\alpha)B(\alpha)^*P(\varepsilon,\alpha)$ have nonpositive real part for every $\varepsilon\geq 0$ and every $\alpha\in(0,\alpha_0]$, and thus the solutions of $\dot x=(A(\alpha)-B(\alpha)B(\alpha)^*P(\varepsilon,\alpha))x$ remain uniformly bounded by the norm of their initial condition: there exists $C>0$ such that $\Vert x_2(t;\alpha,u_{\varepsilon,\alpha},x_{1,0},0)\Vert\leq C\Vert x_{1,0}\Vert$ for every $t\geq 0$, every $\varepsilon\in[0,1]$ and every $\alpha\in(0,\alpha_0]$.
Since $A_2(\alpha)$ is uniformly Hurwitz, using \eqref{16:09} and the Cauchy-Schwarz inequality, the lemma follows.
\end{proof}

By the minimization property, we have $(x_{1,0}^*,x_{2,0}^*) P(\varepsilon,\alpha)\begin{pmatrix}x_{1,0}\\x_{2,0}\end{pmatrix}\leq J_{\varepsilon,\alpha}(u_{\varepsilon,\alpha}; x_{1,0},x_{2,0})$ and therefore, using Lemma \ref{lemtech5},
\begin{equation*}
\begin{split}
2 x_{1,0}^* P_3(\varepsilon,\alpha)x_{2,0}
&= (x_{1,0}^*,x_{2,0}^*) P(\varepsilon,\alpha)\begin{pmatrix}x_{1,0}\\x_{2,0}\end{pmatrix}-x_{1,0}^* P_1(\varepsilon,\alpha) x_{1,0}-x_{2,0}^* P_2(\varepsilon,\alpha) x_{2,0}\\
&\leq J_{\varepsilon,\alpha}(u_{\varepsilon,\alpha}; x_{1,0},x_{2,0})-J_{\varepsilon,\alpha}(u_{\varepsilon,\alpha}, x_{1,0},0)+ \underset{\varepsilon\rightarrow 0}{\mathrm{o}}(1)\|x_{2,0}\|^2\\
&= \underset{\varepsilon\rightarrow 0}{\mathrm{o}}(1)\left(\|x_{1,0}\|^2+\|x_{2,0}\|^2\right)
\end{split}
\end{equation*}
uniformly with respect to $\alpha\in(0,\alpha_0]$.
Since $x_{1,0}$ and $x_{2,0}$ are arbitrary, it follows that $P_3(\varepsilon,\alpha) = \underset{\varepsilon\rightarrow 0}{\mathrm{o}}(1)$ uniformly with respect to $\alpha\in(0,\alpha_0]$.

At this step, we have also obtained that the mappings $(\varepsilon,\alpha)\mapsto P_2(\varepsilon,\alpha)$ and $(\varepsilon,\alpha)\mapsto P_3(\varepsilon,\alpha)$ have a continuous extension at $(0,0)$.

\medskip

\emph{Step 3.}
The Riccati equation \eqref{ricca} gives the equation satisfied by $P_1(\varepsilon,\alpha)$:
\begin{multline}\label{ric1}
A_1(\alpha)^* P_1(\varepsilon,\alpha) + P_1(\varepsilon,\alpha) A_1(\alpha) - P_1(\varepsilon,\alpha) B_1(\alpha) B_1(\alpha)^* P_1(\varepsilon,\alpha) -  P_3(\varepsilon,\alpha) B_2(\alpha) B_1(\alpha)^* P_1(\varepsilon,\alpha)\\
- P_1(\varepsilon,\alpha) B_1(\alpha) B_2(\alpha)^* P_3(\varepsilon,\alpha)^* - P_3(\varepsilon,\alpha) B_2(\alpha) B_2(\alpha)^* P_3(\varepsilon,\alpha)^* = - \varepsilon I_\ell .
\end{multline}
%\begin{multline}\label{ric2}
%A_2^* P_2(\varepsilon) + P_2(\varepsilon) A_2 - P_3(\varepsilon) ^* B_1 B_1^* P_3(\varepsilon) - P_2(\varepsilon) B_2B_1^*  P_3(\varepsilon) \\
%            %
%            -P_3(\varepsilon)^* B_1B_2^*  P_2(\varepsilon) - P_2(\varepsilon) B_2 B_2 ^* P_2(\varepsilon) = - \varepsilon I_m  ,
%\end{multline}
%\begin{multline}\label{ric3}
%A_1^*  P_3(\varepsilon) + P_3(\varepsilon) A_2 - P_1(\varepsilon) B_1 B_1^* P_3(\varepsilon)- P_3(\varepsilon) B_2 B_1 ^* P_3(\varepsilon) \\
%- P_1(\varepsilon) B_1 B_2 ^* P_2(\varepsilon) - P_3(\varepsilon) B_2 B_2 ^* P_2(\varepsilon) = 0 ,
%\end{multline}
Taking the limit $\varepsilon\rightarrow 0$ in \eqref{ric1}, we obtain
\begin{equation}\label{riccaP1}
A_1(\alpha)^* P_1(0,\alpha) + P_1(0,\alpha) A_1(\alpha) - P_1(0,\alpha) B_1(\alpha) B_1(\alpha)^* P_1(0,\alpha) = 0
\end{equation}
which is a Riccati equation with zero weight on the state. Since the pair $(A_1(\alpha),B_1(\alpha))$ satisfies the Kalman condition and thus is stabilizable, and since the Hamiltonian matrix
$$
\begin{pmatrix} A_1(\alpha) & -B_1(\alpha)^* B_1(\alpha) \\ 0 & -A_1(\alpha)^* \end{pmatrix}
$$
has no pure imaginary eigenvalue (because $-A_1(\alpha)$ is Hurwitz), it follows from \cite[Chapter 2, Corollary 2.4.3 page 60 and Theorem 2.4.25 page 77]{Abou-Kandil} that the Riccati equation \eqref{riccaP1} has a (unique) maximal symmetric positive semidefinite solution, which is moreover stabilizing, hence $P_1(0,\alpha)$ is a symmetric positive semidefinite matrix such that $A_1(\alpha)-B_1(\alpha)B_1(\alpha)^* P_1(0,\alpha)$ is Hurwitz, uniformly with respect to $\alpha\in(0,\alpha_0]$.

Moreover, for $\alpha=0$, the above argument shows that \eqref{riccaP1} has a (unique) maximal symmetric positive semidefinite solution $\tilde P$. We claim that $P_1(\varepsilon,\alpha) = \tilde P + \mathrm{o}(1)$ as $(\varepsilon,\alpha)\rightarrow (0,0)$. By \eqref{ric1}, using that $P_2(\varepsilon,\alpha)=\mathrm{o}(1)$ and $P_3(\varepsilon,\alpha)=\mathrm{o}(1)$ as $(\varepsilon,\alpha)\rightarrow (0,0)$, we have
$$
A_1(\alpha)^* P_1(\varepsilon,\alpha) + P_1(\varepsilon,\alpha) A_1(\alpha) - P_1(\varepsilon,\alpha) B_1(\alpha) B_1(\alpha)^* P_1(\varepsilon,\alpha)  = \mathrm{o}(1)
$$
as $(\varepsilon,\alpha)\rightarrow (0,0)$.
Given any sequence $(\varepsilon_k,\alpha_k)$ converging to $(0,0)$, up to some subsequence we must have $P(\varepsilon_k,\alpha_k)\rightarrow \tilde P$ as $k\rightarrow +\infty$ (by uniqueness and maximality in the Riccati theory). Since this argument is valid for any subsequence, it follows that the mapping $(\varepsilon,\alpha)\mapsto P_1(\varepsilon,\alpha)$ has a continuous extension at $(0,0)$.

Finally, let us prove that the closed-loop matrix $A(\alpha) - B(\alpha)B(\alpha)^* P(\varepsilon,\alpha)$, which is written (omitting the dependence in $\varepsilon$ and $\alpha$ to keep a better readability) as
$$
\left(\begin{array}{c|c}
A_1 - B_1 B_1^* P_1 -B_1 B_2^* P_3 & -B_1 B_1^* P_3 - B_1 B_2^* P_2 \\ [1mm]\hline\\[-3mm]
- B_2 B_1^* P_1 -B_2 B_2^* P_3 & A_2 -B_2 B_1^* P_3 - B_2 B_2^* P_2
\end{array}\right)
$$
is uniformly Hurwitz with respect to $(\varepsilon,\alpha)$ small enough. By the Riccati theory, we already know that, for every fixed $\varepsilon>0$, $A(\alpha) - B(\alpha)B(\alpha)^* P(\varepsilon,\alpha)$ is Hurwitz, uniformly with respect to $\alpha\in[0,\alpha_0]$. Now, thanks to \eqref{a.54},
$$
A(\alpha) - B(\alpha)B(\alpha)^* P(\varepsilon,\alpha) = \left(\begin{array}{c|c}
A_1(\alpha) - B_1(\alpha) B_1(\alpha)^* P_1(0,\alpha) + \underset{\varepsilon\rightarrow 0}{\mathrm{o}}(1) & \underset{\varepsilon\rightarrow 0}{\mathrm{o}}(1) \\ [1mm]\hline\\[-3mm]
- B_2(\alpha) B_1(\alpha)^* P_1(0,\alpha) + \underset{\varepsilon\rightarrow 0}{\mathrm{o}}(1) & A_2(\alpha) +\underset{\varepsilon\rightarrow 0}{\mathrm{o}}(1)
\end{array}\right)
$$
where the remainder terms are uniform with respect to $\alpha\in[0,\alpha_0]$.
Using the fact that the determinant is differentiable, and that $d\det(X).H=\mathrm{tr}(\mathrm{com}(X)^* H)$, we get that
\begin{multline*}
\det \left( A(\alpha) - B(\alpha)B(\alpha)^* P(\varepsilon,\alpha) - \lambda I_{m} \right) \\
= \det \left( A_1(\alpha) - B_1(\alpha) B_1(\alpha)^* P_1(0,\alpha) - \lambda I_\ell \right) \, \det \left( A_2(\alpha)-\lambda I_{m-\ell} \right) + \underset{\varepsilon\rightarrow 0}{\mathrm{o}}(1)
\end{multline*}
from which it follows that $A(\alpha) - B(\alpha)B(\alpha)^* P(\varepsilon,\alpha)$ is uniformly Hurwitz with respect to $(\varepsilon,\alpha)\in[0,\varepsilon_0]\times[0,\alpha_0]$ for some $\varepsilon_0>0$ small enough. The proposition is proved.
\end{proof}

\subsection{Asymptotic result in the general case}\label{app:riccati2}
We now drop the assumption \eqref{A3A4zero} and we provide a generalization of Proposition \ref{prop_riccati1}.
Since $A_3(\alpha)=\underset{\alpha\rightarrow 0}{\mathrm{o}}(1)$ and $A_4(\alpha)=\underset{\alpha\rightarrow 0}{\mathrm{o}}(1)$, the control system \eqref{app:sysalpha} can be viewed as a perturbation of the system \eqref{app:sysalpha} under the assumption \eqref{A3A4zero}, which has been studied in Appendix \ref{app:riccati1}.

\begin{proposition}\label{prop_riccati2}
We have the following results:
\begin{itemize}
\item For every $\alpha\in(0,\alpha_0]$, $P(\varepsilon,\alpha)$ is continuous with respect to $\varepsilon\geq 0$. Moreover, the mapping $(\varepsilon,\alpha)\mapsto P(\varepsilon,\alpha)$ has a continuous extension at $(0,0)$, with $P_2(0,0)=0$ and $P_3(0,0)=0$:
\begin{equation}\label{a.542}
P_1(\varepsilon,\alpha) = P_1(0,0) + \mathrm{o}(1) ,\qquad
P_2(\varepsilon,\alpha) = \mathrm{o}(1),\qquad
P_3(\varepsilon,\alpha) = \mathrm{o}(1)
\end{equation}
as $(\varepsilon,\alpha)\rightarrow (0,0)$.
\item For every $\alpha\in[0,\alpha_0]$, the matrix $P_1(0,\alpha)$ is symmetric positive semidefinite (but not necessarily definite). % and the matrix $A_1-B_1B_1^* P_1(0,\alpha)$ is Hurwitz uniformly with respect to $\alpha\in(0,\alpha_0]$.
\item The matrix $A(\alpha) - B(\alpha)B(\alpha)^* P(\varepsilon,\alpha)$ is uniformly Hurwitz with respect to $(\varepsilon,\alpha)$ small enough, meaning that there exist $\varepsilon_0>0$, $\alpha_0'\in(0,\alpha_0]$ and $\eta>0$ such that, for every $(\varepsilon,\alpha)\in[0,\varepsilon_0]\times[0,\alpha_0']$, every eigenvalue $\lambda$ of $A(\alpha) - B(\alpha)B(\alpha)^* P(\varepsilon,\alpha)$ is such that $\Real(\lambda)\leq-\eta$.
%Moreover, $-\eta$ can be chosen as any negative real number which is greater than the maximum of the spectral abscissa of $A_1-B_1B_1^* P_1(0)$ and of $A_2$.
\end{itemize}
\end{proposition}

Note that, as a consequence of \eqref{a.542}, the feedback matrix $K(\varepsilon,\alpha)=-B(\alpha)^* P(\varepsilon,\alpha)$ satisfies
\begin{equation}\label{a.55}
\begin{split}
K(\varepsilon,\alpha) &= -(B_1(\alpha)^*P_1(\varepsilon,\alpha)+B_2(\alpha)^*P_3(\varepsilon,\alpha), B_1(\alpha)^*P_3(\varepsilon,\alpha)+B_2(\alpha)^*P_2(\varepsilon,\alpha)) \\
&= -(B_1(\alpha)^*P_1(0,0),0) + \mathrm{o}(1)
\end{split}
\end{equation}
as $(\varepsilon,\alpha)\rightarrow 0$. In other words, for $(\varepsilon,\alpha)$ small the feedback matrix $K(\varepsilon,\alpha)$ essentially acts on the $\ell$ first modes that are instable (eigenvalues of $A_1(\alpha)$).

\begin{proof}
First of all, proceeding as in the first step of the proof of Proposition \ref{prop_riccati1} is the same, we obtain the continuity of $P(\varepsilon,\alpha)$ with respect to $\varepsilon$. Anyway at this step we do not know yet that $(\varepsilon,\alpha)\mapsto P(\varepsilon,\alpha)$ has a continuous extension at $(0,0)$.

The proof is now different from the second and third steps of the proof of Proposition \ref{prop_riccati1} in which we used the specific (block diagonal) form of the system under the assumption \eqref{A3A4zero}, but we are going to use the result of Proposition \ref{prop_riccati1}.

For the block diagonal control system studied in Section \ref{app:riccati1}, which is written as
\begin{equation}\label{app:sysalphadiag}
\dot y = A^{\textrm{diag}}(\alpha) y + B(\alpha) v
\end{equation}
with
$$
A^{\textrm{diag}}(\alpha) = \begin{pmatrix} A_1(\alpha) & 0 \\ 0 & A_2(\alpha) \end{pmatrix}
$$
we denote in what follows by
$$
Q(\varepsilon,\alpha) = \begin{pmatrix}
Q_1(\varepsilon,\alpha) & Q_3(\varepsilon,\alpha) \\ Q_3(\varepsilon,\alpha)^* & Q_2(\varepsilon,\alpha)
\end{pmatrix}
$$
(instead of $P$) the Riccati matrix solution of \eqref{ricca}. It satisfies the conclusions of Proposition \ref{prop_riccati1}.

We are now going to compare $P(\varepsilon,\alpha)$ (attached to the complete system \eqref{app:sysalpha}) with $Q(\varepsilon,\alpha)$ (attached to the block diagonal system \eqref{app:sysalphadiag}) .

\begin{lemma}\label{lemexp18:06}
There exist $\varepsilon_0>0$ and $\alpha_0'\in(0,\alpha_0]$ such that:
\begin{itemize}
\item $A^{\textrm{diag}}(\alpha)-B(\alpha)B(\alpha)^*Q(\varepsilon,\alpha)$ and $A(\alpha)-B(\alpha)B(\alpha)^*Q(\varepsilon,\alpha)$ are uniformly Hurwitz with respect to $(\varepsilon,\alpha)\in[0,\varepsilon_0]\times[0,\alpha_0']$.
\item We have
\begin{equation}\label{diffexp}
e^{ t (A(\alpha)-B(\alpha)B(\alpha)^*Q(\varepsilon,\alpha)) } = e^{ t (A^{\textrm{diag}}(\alpha)-B(\alpha)B(\alpha)^*Q(\varepsilon,\alpha)) } + \underset{\alpha\rightarrow 0}{\mathrm{o}}(1)
\end{equation}
uniformly with respect to $t\geq 0$ and to $\varepsilon\in[0,\varepsilon_0]$.
\end{itemize}
\end{lemma}

\begin{proof}
The Hurwitz property of $A^{\textrm{diag}}(\alpha)-B(\alpha)B(\alpha)^*Q(\varepsilon,\alpha)$ has been obtained in Proposition \ref{prop_riccati1}.
Since $A_3(\alpha)=\underset{\alpha\rightarrow 0}{\mathrm{o}}(1)$ and $A_4(\alpha)=\underset{\alpha\rightarrow 0}{\mathrm{o}}(1)$, we have $A(\alpha) = A^{\textrm{diag}}(\alpha) + \underset{\alpha\rightarrow 0}{\mathrm{o}}(1)$ and the Hurwitz property of $A(\alpha)-B(\alpha)B(\alpha)^*Q(\varepsilon,\alpha)$ follows for every $\alpha$ small enough.

For the second point, we note that, by the Duhamel formula,
$$
e^{ t (A-BB^*Q) } = e^{ t (A^{\textrm{diag}}-BB^*Q) } + \int_0^t e^{ (t-s) ( A^{\textrm{diag}}-BB^*Q ) } (A - A^{\textrm{diag}}) e^{ s (A-BB^*Q) }\, ds
$$
where we have dropped the dependence in $\varepsilon$ and $\alpha$ for readability. Since the exponentials in the integral involve Hurwitz matrices, and since the middle term in the integral is a $\underset{\alpha\rightarrow 0}{\mathrm{o}}(1)$, the estimate \eqref{diffexp} follows.
\end{proof}

By Lemma \ref{lemexp18:06}, the matrix $A^{\textrm{diag}}(\alpha)-B(\alpha)B(\alpha)^*Q(\varepsilon,\alpha)$ (resp., $A(\alpha)-B(\alpha)B(\alpha)^*Q(\varepsilon,\alpha)$) is uniformly Hurwitz with respect to $(\varepsilon,\alpha)\in[0,\varepsilon_0]\times[0,\alpha_0']$. Hence the feedback control
$$
v_{\varepsilon,\alpha} = -B(\alpha)^* Q(\varepsilon,\alpha) y \qquad \textrm{(resp., $u_{\varepsilon,\alpha} = -B(\alpha)^* Q(\varepsilon,\alpha) x$)}
$$
stabilizes the block diagonal system \eqref{app:sysalphadiag} (resp., the system \eqref{app:sysalpha}) and we infer that there exist $C_1>0$ and $C_2>0$ such that
\begin{equation*}
\Vert y(t;\alpha,v_{\varepsilon,\alpha},x_0)\Vert^2 + \Vert v_{\varepsilon,\alpha}(t)\Vert^2 + \Vert x(t;\alpha,u_{\varepsilon,\alpha},x_0)\Vert^2 + \Vert u_{\varepsilon,\alpha}(t)\Vert^2 \leq C_1e^{-C_2 t}\Vert x_0\Vert^2
\end{equation*}
for every $x_0\in\R^m$, uniformly with respect to $t\geq 0$ and to $(\varepsilon,\alpha)\in[0,\varepsilon_0]\times[0,\alpha_0']$.
For every $\delta>0$, let $T_\delta>0$ be such that $\int_{T_\delta}^{+\infty} C_1e^{-C_2 t}\, dt\leq \delta$. Then
\begin{equation}\label{19:34}
\int_{T_\delta}^{+\infty} \left( \Vert y(t;\alpha,v_{\varepsilon,\alpha},x_0)\Vert^2 + \Vert v_{\varepsilon,\alpha}(t)\Vert^2 + \Vert x(t;\alpha,u_{\varepsilon,\alpha},x_0)\Vert^2 + \Vert u_{\varepsilon,\alpha}(t)\Vert^2 \right) dt \leq \delta \Vert x_0\Vert^2
\end{equation}
for every $x_0\in\R^m$, for every $t\geq 0$ and for all $(\varepsilon,\alpha)\in[0,\varepsilon_0]\times[0,\alpha_0']$.

Besides, since by definition
\begin{equation*}
\begin{split}
y(t;\alpha,v_{\varepsilon,\alpha},x_0) &= e^{ t (A^{\textrm{diag}}(\alpha)-B(\alpha)B(\alpha)^*Q(\varepsilon,\alpha)) } x_0 \\
x(t;\alpha,u_{\varepsilon,\alpha},x_0) &= e^{ t (A(\alpha)-B(\alpha)B(\alpha)^*Q(\varepsilon,\alpha)) } x_0
\end{split}
\end{equation*}
we infer from Lemma \ref{lemexp18:06} that
$$
\Vert y(t;\alpha,v_{\varepsilon,\alpha},x_0)\Vert^2 + \Vert v_{\varepsilon,\alpha}(t)\Vert^2 = \Vert x(t;\alpha,u_{\varepsilon,\alpha},x_0)\Vert^2 +  \Vert u_{\varepsilon,\alpha}(t)\Vert^2 + \underset{\alpha\rightarrow 0}{\mathrm{o}}(1) \Vert x_0\Vert^2
$$
uniformly with respect to $t\geq 0$ and to $\varepsilon\in[0,\varepsilon_0]$. Hence, combining with \eqref{19:34}, we get that
\begin{equation}\label{19:49}
\begin{split}
J^\textrm{diag}_{\varepsilon,\alpha}(v_{\varepsilon,\alpha},x_0) &= \int_0^{+\infty} \left( \varepsilon\Vert y(t;\alpha,v_{\varepsilon,\alpha},x_0)\Vert^2 + \Vert v_{\varepsilon,\alpha}(t)\Vert^2 \right) dt \\
&= \int_0^{+\infty} \left(  \varepsilon\Vert x(t;\alpha,u_{\varepsilon,\alpha},x_0)\Vert^2 + \Vert u_{\varepsilon,\alpha}(t)\Vert^2 \right) dt + \underset{\alpha\rightarrow 0}{\mathrm{o}}(1) \Vert x_0\Vert^2 \\
&= J_{\varepsilon,\alpha}(u_{\varepsilon,\alpha},x_0) + \underset{\alpha\rightarrow 0}{\mathrm{o}}(1) \Vert x_0\Vert^2
\end{split}
\end{equation}
uniformly with respect to $\varepsilon\in[0,\varepsilon_0]$.
Here, with evident notations, $J^\textrm{diag}_{\varepsilon,\alpha}$ stands for the cost functional \eqref{121960} attached to the block diagonal control system \eqref{app:sysalphadiag}.
Since $Q(\varepsilon,\alpha)$ is the solution of the Riccati equation, $J^\textrm{diag}_{\varepsilon,\alpha}(v_{\varepsilon,\alpha},x_0)$ is the optimal cost for the diagonal block system \eqref{app:sysalphadiag} and we have
$$
J^\textrm{diag}_{\varepsilon,\alpha}(v_{\varepsilon,\alpha},x_0) = x_0^*Q(\varepsilon,\alpha)x_0 .
$$
Besides, $J_{\varepsilon,\alpha}(u_{\varepsilon,\alpha},x_0)$ may not be the optimal cost for the system \eqref{app:sysalpha} because $u_{\varepsilon,\alpha}$ is defined with the matrix $Q(\varepsilon,\alpha)$ which may differ from $P(\varepsilon,\alpha)$. Anyway, since $x(t;\alpha,u_{\varepsilon,\alpha},x_0)$ and $u_{\varepsilon,\alpha}$ are linear in $x_0$, there exists a symmetric positive semidefinite matrix
$$
R(\varepsilon,\alpha) = \begin{pmatrix} R_1(\varepsilon,\alpha) & R_3(\varepsilon,\alpha) \\ R_3(\varepsilon,\alpha)^* & R_4(\varepsilon,\alpha) \end{pmatrix}
$$
such that
$$
J_{\varepsilon,\alpha}(u_{\varepsilon,\alpha},x_0)= x_0^* R(\varepsilon,\alpha) x_0 .
$$
We infer from \eqref{19:49} that
$$
x_0^*Q(\varepsilon,\alpha)x_0 = x_0^* R(\varepsilon,\alpha) x_0 + \underset{\alpha\rightarrow 0}{\mathrm{o}}(1) \Vert x_0\Vert^2
$$
for every $x_0\in\R^m$, uniformly with respect to $\varepsilon\in[0,\varepsilon_0]$. Applying Proposition \ref{prop_riccati1}, we get that
$$
R_1(\varepsilon,\alpha)=Q_1(0,0) + \mathrm{o}(1),\qquad
R_2(\varepsilon,\alpha)=\mathrm{o}(1),\qquad R_3(\varepsilon,\alpha)=\mathrm{o}(1)
$$
as $(\varepsilon,\alpha)\rightarrow(0,0)$, and by the arguments of Step 3 in the proof of Proposition \ref{prop_riccati1}, $R_1(0,0)=Q_1(0,0)$ is the (unique) maximal symmetric positive semidefinite matrix solution of the Riccati equation
\begin{equation}\label{riccaX}
A_1(0)^* X + X A_1(0) - X B_1(0) B_1(0)^* X = 0 .
\end{equation}

By the minimization property of the Riccati matrix for the control system \eqref{app:sysalpha}, we have $
x_0^*P(\varepsilon,\alpha)x_0 \leq J_{\varepsilon,\alpha}(u_{\varepsilon,\alpha},x_0)= x_0^* R(\varepsilon,\alpha) x_0$ and therefore
\begin{equation}\label{09:26}
P_2(\varepsilon,\alpha)=\mathrm{o}(1)\qquad\textrm{and}\qquad P_3(\varepsilon,\alpha)=\mathrm{o}(1)
\end{equation}
as $(\varepsilon,\alpha)\rightarrow(0,0)$ and $P_1(\varepsilon,\alpha)\leq R_1(\varepsilon,\alpha)$ in the sense of symmetric positive semidefinite matrices.
In particular, $P_1(0,0)\leq R_1(0,0)=Q_1(0,0)$.

It remains to prove that $P_1(\varepsilon,\alpha)=P_1(0,0) + \mathrm{o}(1)$ as $(\varepsilon,\alpha)\rightarrow(0,0)$.
The Riccati equation \eqref{ricca} gives the equation satisfied by $P_1(\varepsilon,\alpha)$:
\begin{equation*}%\label{ric09:23}
A_1^*P_1 + A_4^*P_3^* + P_1A_1 + P_3 A_4 - P_1B_1B_1^*P_1 - P_1B_1B_2^*P_3^* - P_3B_2B_1^*P_1 - P_3B_2B_2^*P_3^* = - \varepsilon I_\ell
\end{equation*}
where we have dropped the dependence in $\varepsilon$ and $\alpha$ for readability.
Using \eqref{09:26} and the fact that $A_4(\alpha)=\underset{\alpha\rightarrow 0}{\mathrm{o}}(1)$, we obtain
\begin{equation}\label{ric09:31}
A_1(\alpha)^*P_1(\varepsilon,\alpha) + P_1(\varepsilon,\alpha)A_1(\alpha) - P_1(\varepsilon,\alpha)B_1(\alpha)B_1(\alpha)^*P_1(\varepsilon,\alpha) = \mathrm{o}(1)
\end{equation}
as $(\varepsilon,\alpha)\rightarrow(0,0)$. Noting that $P_1(\varepsilon,\alpha)\leq R_1(\varepsilon,\alpha)$ and thus is bounded, given any sequence $(\varepsilon_k,\alpha_k)$ converging to $(0,0)$, up to some subsequence we have $P(\varepsilon_k,\alpha_k)\rightarrow \tilde P$ as $k\rightarrow +\infty$ with $\tilde P$ which is, thanks to \eqref{ric09:31}, the maximal symmetric positive semidefinite matrix solution of \eqref{riccaX} (as in Step 3 in the proof of Proposition \ref{prop_riccati1}). By uniqueness, it follows that $\tilde P = P_1(0,0)$. Since this argument is valid for any subsequence, the conclusion follows.

The proposition is proved.
\end{proof}

\end{document}